\documentclass[11pt,a4paper, envcountsame,mathserif]{amsart}
\usepackage[usenames,dvipsnames]{color}

\usepackage[utf8]{inputenc}
\usepackage{pdfpages}

 \usepackage[colorlinks,citecolor=blue,urlcolor=blue, linkcolor=blue, backref]{hyperref}

\usepackage[capitalise]{cleveref}		

 \usepackage{amsmath, amssymb, xspace}
 \usepackage{graphicx}

 \usepackage[all]{xy}

\usepackage{tikz}
\usetikzlibrary{arrows,snakes,positioning,backgrounds,shadows}
\usepackage{stmaryrd}

\newtheorem{theorem}{Theorem}[section]

\newtheorem{thm}[theorem]{Theorem}

\newtheorem{fact}[theorem]{Fact}

\newtheorem{proposition}[theorem]{Proposition}

\newtheorem{prop}[theorem]{Proposition}

\newtheorem{claim}[theorem]{Claim}

\newtheorem{conjecture}[theorem]{Conjecture}

\newtheorem{lemma}[theorem]{Lemma}		

\newtheorem{corollary}[theorem]{Corollary}

\newtheorem{cor}[theorem]{Corollary}

\newtheorem{question}[theorem]{Question}

\theoremstyle{definition}
\newtheorem{definition}[theorem]{Definition}

\newtheorem{remark}[theorem]{Remark}

\newcommand{\KO}{\mathcal{O}}
\newcommand{\DII}{\Delta^0_2}
\newcommand{\NN}{{\mathbb{N}}}

\newcommand{\ZZ}{{\mathbb{Z}}}
\newcommand{\sub}{\subseteq}
\newcommand{\sN}[1]{_{#1\in \NN}}
\newcommand{\uhr}[1]{\! \upharpoonright_{#1}}

\newcommand{\SI}[1]{\Sigma^0_{#1}}
\newcommand{\PI}[1]{\Pi^0_{#1}}

\newcommand{\bi}{\begin{itemize}}
\newcommand{\ei}{\end{itemize}}
\newcommand{\bc}{\begin{center}}
\newcommand{\ec}{\end{center}}

\newcommand{\Halt}{{\ES'}}
\newcommand{\ES}{\emptyset}

\newcommand{\tp}[1]{2^{#1}}
\newcommand{\ex}{\exists}
\newcommand{\fa}{\forall}

\newcommand{\la}{\langle}
\newcommand{\ra}{\rangle}

\newcommand{\seqcantor}{2^{ \NN}}

\newcommand{\cantor}{\seqcantor}

\newcommand{\Opcl}[1]{[#1]^\prec}
\newcommand{\leT}{\le_{\mathrm{T}}}

\newcommand{\MLR}{\mbox{\rm \textsf{MLR}}}

\newcommand{\n}{\noindent}

\newcommand{\vsp}{\vspace{6pt}}
\newcommand{\leb}{\mathbf{\lambda}}

\newcommand{\sss}{\sigma}
\newcommand{\aaa}{\alpha}

\newcommand{\set}[2]{\{#1 \colon #2\}}

\newcommand \seq[1]{{\left\langle{#1}\right\rangle}}

\newcommand\+[1]{\mathcal{#1}}

\newcommand{\wt}{\widetilde}
\newcommand{\ol}{\overline}

\newcommand{\ape}{\, \hat{\ } \, }

\newcommand{\lra}{\leftrightarrow}
\newcommand{\LR}{\Leftrightarrow}
\newcommand{\RA}{\Rightarrow}

\newcommand{\DA}{\downarrow}

\newcommand{\sssl}{\ensuremath{|\sigma|}}

\newcommand{\range}{\ensuremath{\mathrm{range}}}

\def\uh{\upharpoonright}

  \DeclareMathOperator{\Img}{Im}

  \newcommand{\CR}{\mbox{\rm \textsf{CR}}}

\newcommand{\frb}{\mathfrak{b}}

\newcommand{\fra}{\mathfrak{a}}
\newcommand{\fru}{\mathfrak{u}}

\numberwithin{equation}{section}
\renewcommand{\hat}{\widehat}
\begin{document}

\title{Logic Blog 2019}

 \author{Editor: Andr\'e Nies}

\maketitle


 {
The Logic   Blog is a shared platform for
\bi \item rapidly announcing  results and questions related to logic
\item putting up results and their proofs for further research
\item parking results for later use
\item getting feedback before submission to  a journal
\item foster collaboration.   \ei

Each year's   blog is    posted on arXiv a 2-3 months after the year has ended.
\vsp
\begin{tabbing}

 \href{http://arxiv.org/abs/1902.08725}{Logic Blog 2018} \ \ \ \   \= (Link: \texttt{http://arxiv.org/abs/1902.08725})  \\
 
 \href{http://arxiv.org/abs/1804.05331}{Logic Blog 2017} \ \ \ \   \= (Link: \texttt{http://arxiv.org/abs/1804.05331})  \\
 
 \href{http://arxiv.org/abs/1703.01573}{Logic Blog 2016} \ \ \ \   \= (Link: \texttt{http://arxiv.org/abs/1703.01573})  \\
 
  \href{http://arxiv.org/abs/1602.04432}{Logic Blog 2015} \ \ \ \   \= (Link: \texttt{http://arxiv.org/abs/1602.04432})  \\
  
  \href{http://arxiv.org/abs/1504.08163}{Logic Blog 2014} \ \ \ \   \= (Link: \texttt{http://arxiv.org/abs/1504.08163})  \\

   \href{http://arxiv.org/abs/1403.5719}{Logic Blog 2013} \ \ \ \   \= (Link: \texttt{http://arxiv.org/abs/1403.5719})  \\

    \href{http://arxiv.org/abs/1302.3686}{Logic Blog 2012}  \> (Link: \texttt{http://arxiv.org/abs/1302.3686})   \\

 \href{http://arxiv.org/abs/1403.5721}{Logic Blog 2011}   \> (Link: \texttt{http://arxiv.org/abs/1403.5721})   \\

 \href{http://dx.doi.org/2292/9821}{Logic Blog 2010}   \> (Link: \texttt{http://dx.doi.org/2292/9821})  
     \end{tabbing}

\vsp

\n {\bf How does the Logic Blog work?}

\vsp

\n {\bf Writing and editing.}  The source files are in a shared dropbox.
 Ask Andr\'e (\email{andre@cs.auckland.ac.nz})  in order    to gain access.

\vsp

\n {\bf Citing.}  Postings can be cited.  An example of a citation is:

\vsp

\n  H.\ Towsner, \emph{Computability of Ergodic Convergence}. In  Andr\'e Nies (editor),  Logic Blog, 2012, Part 1, Section 1, available at
\url{http://arxiv.org/abs/1302.3686}.}

%

%

 
The logic blog,  once it is on  arXiv,  produces citations on Google Scholar.
%
\newpage
\tableofcontents

\part{Computability theory and randomness}

 \newcommand{\tr}{\mbox{\rm \textsf{Tr}}}

\newcommand{\Malg}[1]{M_{#1}}

\section{Barmpalias and Nies: a  randomness notion between ML and KL-randomness}

Jason Rute \cite[Section 10]{Rute:16} introduced  possible strengthenings of  Kolmo\-gorov-Loveland (KL)-randomness that are still implied by ML-randomness. In fact  he worked in  the general setting of computable probability spaces, where he also defined a notion of computable randomness.  
The following is one of his notions. It relies on  the little explored notion of  partial computable randomness;  see \cite[Ch.\ 7]{Nies:book}.  (A betting strategy can be undefined off the sequences it has to succeeds on, which allows for potentially indefinite waiting before making a bet.) 
\begin{definition} $Z \in \cantor$ is called \emph{Rute random}  if $\Phi(Z)$ is partial computably random for each  measure preserving Turing functional $\Phi$. 

A pair $\+ R = (\Phi, L)$ of a m.p.\ Turing functional and a  partial computable martingale will be called a \emph{Rute betting strategy}. We say that $\+ R$ is total if both $\Phi$ and $L$ are total functions. 
\end{definition}

To say that $\Phi$ is  measure preserving means that $\leb \Phi^{-1}(\+ C) = \leb \+ C$ for each measurable $\+ C \sub \cantor$. It suffices to require this for cylinders $\+ C = \Opcl \sss$ where $\sss$ is any string. Such Turing functionals are almost total: the domain is a conull $\PI 2$ set. So $\Phi(Z)$ is defined for any Kurtz ``random" $Z$. Its  kernel $\{\la X, Y \ra \colon \Phi^X= \Phi^Y \downarrow\}$ is a $\PI 2$ equivalence relation with all classes null.     $\Phi$ is given by a binary tree of  uniformly $\SI 1$ sets $U_\sss = \Phi^{-1}(\Opcl \sss)$ such that $\leb U_\sss = \tp{-\sssl}$. 

If $\Phi$ is total (i.e., a  tt functional) then $\Phi$  is   onto by a  compactness argument. In this case the kernel is $\PI 1$.  An easy example of a m.p.\ total function is the shift functional $T$ where $T(Z)$ is $Z$ with the first bit erased. For a more interesting example, view bit sequences as $2$-adic numbers, and consider a computable unit $r \in \ZZ_2^*$ (anything that ends in $1$). Then $\Phi^Z = rZ$ is a m.p.\ total Turing functional, with inverse $\Psi^Z = r^{-1} Z$.  This example can be generalised to sequences over a $p$-bit alphabet,  for a prime $p$.

Each scan rule in the sense of \cite{Merkle.ea:06, Nies:book} induces a measure preserving tt functional  on Cantor space. So Rute randomness  implies KL-randomness. 

The sets $\Phi^{-1} ( \{ X \colon L(X)> \tp m\})$ form a ML-test that succeeds on each  $Z$ on which  the strategy  does.  So ML-randomness implies Rute randomness.

  \begin{fact} Given a Rute strategy  $\+ R = \la \Phi, L \ra$ with $\Phi$ total but possibly $L$ partial, there  are total Rute strategies $\+ R_0, \+ R_1$ such that if $\+ R$ succeeds on $Z$ then one of the $\+ R_i$ succeeds as well. \end{fact}
  
  \begin{proof} We use the similar  result from  \cite{Merkle.ea:06} that given a KL-strategy $\+ K$   (also see \cite[Ex.\ 7.6.25]{Nies:book}) we can build two total KL-strategies $\+ K_0, \+ K_1$ that together cover the success set of $\+ K$.

  In our case $\+ K$ is simply $L$ (with the identity scan rule). We obtain scan rules $\+ K_i = \la \+ S_i, B_i\ra$ Let  $\Phi_i=\+ S_i \circ \Phi$ which are  m.p.\ total Turing functionals. Now the $\+ R_i $ given by $\Phi_i, B_i$ are as needed.
   \end{proof}  

Recall that each KL-betting strategy fails on some c.e.\ set (e.g \cite{Merkle.ea:06}. This argument doesn't quite go through. However the following is easy. 
 
\begin{prop} Each Rute strategy $(\Phi, L)$  fails on a non-Kurtz random. \end{prop}
\begin{proof} We may assume that $\Phi$ is total, else it fails by not being defined. So we can also assume that $L$ is total.  Let $R$ be a computable path along which $L$ is bounded (say in the $n$-th bet, along $R$ the martingale  $L$ goes up by at most $\tp{-n}$; this is a $\SI 1$ event which has to happen on at least one 1-bit extension of a string). Then $\Phi^{-1}(R)$ is a $\PI 1$ null class on which the Rute strategy fails. \end{proof}  
As a consequence, if we wanted to show that Rute = ML, we'd need at least two strategies to cover all non-MLR randoms. 

\begin{question} Determine lowness for Rute randomness. Note that it is contained in $\text{Low} (\MLR, \CR)$ which coincides with  the $K$-trivials \cite{Nies:AM}. \end{question}

\section{Nies: For non-high states, strong Solovay tests are as powerful as general quantum ML tests}

Recall that for a non-high bit sequence $Z$, Schnorr randomness implies ML-randomness (e.g. \cite[Ch.\ 3]{Nies:book}). Here we give a similar argument in the setting of randomness for infinite sequences of qubits \cite{Nies.Scholz:18}.  

We use the convention as in the final section of \cite{Nies.Scholz:18} that all states and projections are elementary, i.e.\ the relevant matrices consist of algebraic complex numbers. So it is OK to directly use a state $\rho$ as an oracle. (In general we would have to work with elementary approximations to the density matrices $\rho\uhr n$, similarly as is done in computable analysis to deal with, say, sequences of reals).

\begin{definition}[\cite{Nies.Scholz:18}, 3.14] \label{df:qSL}  \mbox{}
\bi \item A \emph{quantum Solovay  test}    is an effective  sequence $\seq {G_r}\sN r$ of   quantum $\SI 1$ sets such that $\sum_r \tau(G_r) < \infty$.
  \item We say that the test is \emph{strong} if the $G_r$ are given as projections; that is,  from $r$ we can compute $n_r$ and a   matrix of algebraic numbers  in   $\Malg {n_r}$ describing $G_r =  p_r$. 
%
\item For  $\delta \in (0,1)$, we     say that $\rho$ \emph{fails} the quantum  Solovay  test \emph{at order} $\delta$   if   $ \rho(G_r) > \delta$ for infinitely many  $r$; otherwise $\rho$ \emph{passes} the qML test \emph{at order} $\delta$. 
\item We say that $\rho$ is \emph{quantum Solovay-random}  if it passes each quantum Solovay  test  $\seq {G_r}\sN r$ at each positive order, that is,  $\lim_r \rho(G_r)=0$.  \ei
 \end{definition}

Tejas Bhojraj showed that q-Solovay tests are equivalent to qML tests (but the order at which  the test succeeds changes). However, it is likely that strong q-Solovay tests are indeed more restricted, hence the significance of the following \emph{Resultatchen}.

\begin{prop} Let $\rho$ be a non-high state on the usual CAR algebra $M_{ \infty}$. If $\rho$ fails some qML test $\seq {G^r} \sN r$  at order $\delta >0$, then it fails a strong Solovay test  $\seq {q_r}$ at the same order. Moreover we can achieve that $\sum_r \tr(q_r)$ is a computable real.  \end{prop}
  
  \begin{proof} 
Given a qML-test  $\seq {G^r} = \seq {p^r_n}$ as defined in \cite[3.5]{Nies.Scholz:18} that  $\rho$ fails at order $\delta$, the  function $f(r) = \mu n . \tr (\rho \uhr n p^r_n) > \delta$ is total and satisfies  $f \le_T \rho$. Since $\rho$ is not high, there is a computable function $h$ such that $\ex^\infty r . h(r) \ge f(r)$. Let now $q_r = p^{r}_{h(r)}$. This is a strong Solovay test of the required form, and $\rho$ fails it at the same order $\delta$. \end{proof}

     \section{Franklin and Turetsky: (non)convexity of difference random degrees}

The ML-random degrees can be divided into two sets: those which compute $\mathbf{0'}$, and those which do not compute a PA-degree.  The latter correspond to the difference random degrees.  The first set is known to be convex: if $\mathbf{d}_0$ and $\mathbf{d}_1$ are MLR degrees above $\mathbf{0'}$, then every degree $\mathbf{d}_2$ with $\mathbf{d}_0 \le \mathbf{d}_2 \le \mathbf{d}_1$ is an MLR degree.  This is because every degree above $\mathbf{0'}$ contains a random.  Still, it's natural to ask whether convexity holds for the difference random degrees.  Here we construct a counterexample.  Our construction simply requires combining several existing results.

First we need the following:

\begin{proposition}
Every $\Delta^0_2$ MLR degree computes a $K$-trivial which does not have c.e.\ degree.
\end{proposition}

One could prove this directly, or just invoke Ku\v{c}era's result that every $\Delta^0_2$ MLR degree bounds a noncomputable c.e.\ degree, Hirschfeldt et al.'s \cite{Hirschfeldt.Nies.ea:07} results showing that said degree must be $K$-trivial, and then Yates's result that every noncomputable c.e.\ degree bounds a minimal degree (and then Nies \cite{Nies:AM}     that $K$-triviality is downwards closed under Turing reducibilty).  And I guess Sacks's density theorem to get that a minimal degree is not a c.e.\ degree.

Next, we need the following result of Greenberg and Turetsky (see here: \url{https://arxiv.org/abs/1707.00258}):

\begin{proposition}
For every $K$-trivial $A$ there is a c.e.\ $K$-trivial $B \ge_T A$ such that every ML-random which computes $A$ computes $B$.
\end{proposition}
Of course, if $A$ does not have c.e.\ degree, then $B$ will be strictly above $A$.  Finally we'll need this:

\begin{proposition}
If $A$ does not compute $B$, then for almost every oracle $Z$, $A \oplus Z$ does not compute $B$.
\end{proposition}

\begin{proof}
Suppose otherwise.  Fix an $e$ such that for positive measure of oracles $Z$, $\Phi_e(A\oplus Z) = B$.  Apply Lebesgue density and then majority vote to show that $A$ computes $B$, contrary to hypothesis.
\end{proof}

We'll also need van Lambalgen's theorem.  We're ready to build the counterexample.

\smallskip

Fix $X$ an incomplete $\Delta^0_2$ random.  Fix $A \le_T X$ a $K$-trivial which does not have c.e.\ degree.  Fix $B >_T A$ such that every ML-random computing $A$ computes $B$.  For almost every oracle $Y$, $X \oplus Y$ is random, $X \oplus Y$ does not compute $\mathbf{0'}$, and $A \oplus Y$ does not compute $B$, so fix such a $Y$.

Our failure of convexity is provided by $Y \le_T A \oplus Y \le_T X \oplus Y$.  Suppose $A \oplus Y$ were Turing equivalent to a random $Z$.  Then $A\oplus Y \ge_T Z \ge_T A \oplus Y$, and thus $Z \ge_T A$.  Then $Z \ge_T B$, by choice of $B$.  So $A \oplus Y \ge_T B$, contrary to choice of $Y$.

\smallskip

It's known that convexity holds for hyperimmune-free MLR degrees. In fact, if $B$ is ML-random and of hyperimmune free degree, and $A \le_T B$ is noncomputable, then $A$ is of ML-random Turing degree using that  $A \le_{tt}B$ and an argument of Demuth. How common are failures outside those?  Does every hyperimmune MLR degree bound a failure of convexity?  Is every pair of hyperimmune MLR degrees with one strictly above the other a failure of convexity?

 
%

  \part{Group  theory and its connections to  logic}
    \section{Nies and Segal: Non-axiomatizability for classes of profinite groups}
 
  \newcommand{\UT}{\ensuremath{\mathrm{UT}_3}}
\newcommand{\Zp}{\ensuremath{\mathbb Z_p}}
\newcommand{\UTp}{\UT(\Zp)}

Expressiveness of first-order  logic for classes of groups remains an interesting  topic   because it  straddles logic and group theory and connects these areas  in novel ways. For topological  groups the topic  is particularly interesting, because in first-order  logic  (say in the language   of  groups)  one can't directly talk about open sets:   first-order means that one can only talk about elements of the structure.

 Here we show that some commonplace  properties   cannot be  axiomatised by a single sentence.   

\subsection{Number of generators} 
The following results and proofs work with small changes also in the setting of discrete groups. 
\begin{proposition} \mbox{} 

\bi \item[(a)] Being $d$-generated cannot be expressed by a single first order sentence within the profinite groups. 

\item[(b)] Being finitely generated cannot be expressed by a single first order sentence within the profinite groups.

\ei  \end{proposition}
\begin{proof}
Let $C_q$ denote the cyclic group of order $q$. By a    result of Szmielev related to her proof that 
 the theory of abelian groups in decidable  (1950s), for each sentence $\phi$  in the language of groups the following holds:  for almost all primes $q$, an abelian group $G$ satisfies $\phi$ iff $G \times C_q$ satisfies $\phi$. 

Recall that $\hat \ZZ$ is the profinite completion of $\ZZ$, which is the Cartesian product of all the $\ZZ_p$ for $p$ a prime. It is 1-generated and has each finite  cyclic group as a quotient (in fact it is the free profinite group of dimension 1).
If $\phi$ expresses being $d$-generated, then $(\hat \ZZ)^d \models \phi$. Let $q$ be
 a prime as above, then also $(\hat \ZZ)^d \times C_q \models \phi$. Yet, this group is not $d$-generated because it has $C_q^{d+1}$ as a quotient..  
 
  A similar argument   shows that no sentence $\phi$ can express being finitely generated among profinite groups. Modify the argument above. Replace   $(\hat \ZZ)^d \times C_q  $ by the profinite group $\hat  \ZZ \times (C_q)^{\aleph_0}$ which is not finitely generated but would also satisfy $\phi$, for almost every prime $q$.  
This uses that $\phi$ can be expressed as a Boolean combination   of so-called Szmielev invariant sentences, which can only contain information about finitely many primes; see Hodges \cite[Thm A.2.7]{Hodges:93}.
\end{proof}

A profinite group $G$ has \emph{rank} $r$ if each closed subgroup of $G$ is $r$-generated. For an abelian profinite group $G$, this is equivalent to  the whole group being $r$-generated,  by the structure theorem that $G$ is a direct product of pro-cyclic groups and finite cyclic groups. This implies:
\begin{corollary} There is no first-order sentence   expressing that  a profinite  group $G$  has finite rank.
\end{corollary}
On the other hand, the entire theory of a profinite group $G$ contains the information whether it is (topologically)  finitely generated: Jarden and Lubotzky~\cite{Jarden.Lubotzky:08} (also  see \cite[Thm.\ 4.2.3]{Segal:09}) proved that if $G,H$ are profinite groups, $G$ is f.g.\ and $G \equiv H$ then $G \cong H$.      
 Whether a profinite group is  $d$-generated  {can} also be recognized  from the entire theory, because it means that each finite quotient is $d$-generated, and one can express in f.o.\ logic what the finite quotients are again by  Jarden and Lubotzky~\cite{Jarden.Lubotzky:08}.
 
 This leaves open the question whether a uniform  \emph{set} of axioms  can express that a  profinite group is f.g.
 \begin{question} Is there a set of sentences $S$  in the language of group theory such that for profinite groups $G$,  we have  $G \models S \LR  G$ is f.g.? \end{question}
 
 \subsection{Comparison with the discrete case}
 Let's compare this for a moment with the case of  discrete groups where f.g.\ now means literally finitely generated.
 
  The elementary  theory of a group $G$ doesn't  determine whether  $G$  is f.g.: by an easy elementary chain argument, for each countable group $G$ there is a countable group $H$ such that $G$ is an elementary submodel of $H$, in particular has the same theory,  and $H$ is not f.g.\  \cite[Section 4]{Nies:DescribeGroups}.  (This argument using elementary chains of type $\omega$  doesn't work for profinite groups. The reason is that model theoretic constructions such as adding a constant, or union of elementary chains, don't work in the topological setting such as for  profinite structures.)

We note that the  theory $T$  of f.g.\ groups  is  $\Pi^1_1$ complete by Morozov and Nies \cite{Morozov.Nies:05}. Its    models could be called ``pseudo-f.g.", in analogy with  the pseudofinite groups. An example of such a group that is not f.g.\ is  the free group of infinite dimension $F_\omega$.
   
\subsection{Finite axiomatisability  implies that the set of   primes that occur is finite}
In the following ``FA" is short  for ``finitely axiomatisable".

We begin with an observation  of T.\ Scanlon. Rings will be commutative with unit element.  $\ZZ_p$ denotes the ring (or sometimes group) of $p$-adic integers\footnote{An fun   exercise to get used to the  3-adic integers  is the following. Note that   $-1 = \ldots 22222222222$, and verify that ${-1}^2= 1$ in $\ZZ_3$ using   the multipication algorithm from elementary school.}.
\begin{proposition}[going back to Scanlon, 2016] Let $R $ be the profinite ring $ \prod_{p \in S} \ZZ_p$ where $\pi$ is an  infinite set of  primes.  Then $R$ is not FA within the class of profinite rings. \end{proposition}

 \begin{proof} This follows from the  Feferman-Vaught theorem (see any large text on model theory, such as \cite{Hodges:93}).    For every sentence $\phi$ in the language of rings there is a finite sequence of sentences $\psi_1, \ldots, \psi_n$ in the language of rings and a formula $\theta(x_1,\ldots,x_n)$ in the language of Boolean algebras so that for any index set $I$ and any family of rings $R_i$ indexed by $I$, letting    \bc $X_j := \{ i \in I ~:~ R_i \models \psi_j \}$,  \ec we have  \bc $\prod_{i \in I} R_i \models \phi$ if and only if ${\mathcal P}(I) \models \theta(X_1,\ldots,X_n)$.    \ec Assume  for a contradiction that a sentence  $\phi$  shows that $R$ is FA.    By the pigeon hole principle, we can find two distinct primes $r \neq q$  in $\pi$ so that for all $j \leq n$ we have ${\mathbb Z}_r \models \psi_j \Longleftrightarrow {\mathbb Z}_q \models \psi_j$.   Define $R_p := {\mathbb Z}_p$ if $p \neq r$ and $R_r := {\mathbb Z}_q$.   Then $R' := \prod R_p \models \phi$ but $R' \not \cong  R$ as, for example, $r$ is a unit in $R'$ but not in $R$.
 \end{proof}
 $\UT(R)$ denotes the usual Heisenberg group over the ring $R$. It is well known that $Z(\UT(R)) \cong (R, +)$ because the centre consists of the matrices with only a nonzero entry in the upper right corner  (other than the diagonal).  
\begin{proposition} Let $R $ be the profinite ring $ \prod_{p \in \pi} \ZZ_p$ where $\pi$ is a   set of  primes. 
 If $\UT(R)$ is FA  within the profinite groups then  $\pi$ is finite.  
 \end{proposition}
 \begin{proof}
 Note that $\UT(R)$ can be interpreted in $R$ by a collection of first-order formulas that do not depend on $R$. Hence,  for each  sentence $\theta$ in the language of groups, there is a sentence $\wt \theta$ in the language of  rings such that for each ring $R$, 
 \bc $\UT(R) \models \theta 
 \LR R \models \wt \theta$. \ec
 
Assume  for a contradiction that a sentence $\theta $ shows that $\UT(R)$ is FA. within the profinite groups. Applying the above Feferman-Vaught analysis to $\phi = \wt \theta$,   let $r\neq q$ be primes as above, and let also $R'$ defined as above. Then $R' \models \wt \theta$ and hence  $\UT(R')  \models \theta$. However, $\UT(R') \not \cong \UT(R)$ because every element in the centre of $\UT(R')$, but not in the centre of   $\UT(R)$, is divisible by $r$.

\end{proof}

A similar argument    works for some other   algebraic groups. For instance,  for each $n$,  $GL_n(R)$ can be interpreted in $R$ by formulas not depending on $R$, and  its centre consists of  the scalar matrices  $\alpha I_n$ where $\alpha$ is a   unit of $R$.  As mentioned,  the groups of units  are not isomorphic for the rings $R'$ and $R$ as above.

\section{Nies, Perin and Segal:  all f.g.\ profinite groups are  $\omega$-homogeneous}

Let $G$ be a profinite group (separable, as always) and $\ol g \in G^k$ where $k \in \NN$. We want to show that the type of $\ol g$ (i.e., the first-order theory of $(G, \ol g)$) determines the orbit of $\ol g$ in $G$. In the case of countable groups such a result  has been shown e.g.\ for the f.g.\ free groups \cite{Perin.Sklinos:12}.    The following was first noticed in discussions between Segal and   Chloe Perin at an OW meeting in  2015.  Her Master's student Noam Kolodner covered it in his thesis.
\begin{theorem} Let $G, H$ be (topologically) f.g.\ profinite  groups, $k \in \NN$, $\ol g \in G^k$ and $\ol h \in H^k$. 
 If $(G,\ol g) \equiv (H, \ol h) $  then  $(G,\ol g) \cong (H, \ol h) $. \end{theorem}
 Here in   fact it suffices to look at the $\SI 2$ formulas satisfied by $\ol g$ in $G$ to determine the orbit.  The case $k=0$ is due to Jarden and Lubotzky~\cite{Jarden.Lubotzky:08}.

We need a lemma that  is a straightforward extension of \cite[Prop. \ 16.10.7]{Fried.Jarden:06}. Let $\Img(G, \ol g) $ denote  the set  of pairs $(R, \ol r)$, $R$ a finite group, $\ol r \in R^k$ such that there is an epimorphism $(G, \ol g) \rightsquigarrow (R, \ol r)$, i.e.\ it sends $g_i$ to $r_i$ for $i<k$. (The symbol $\rightsquigarrow$ will denote such epimorphisms.)
 For a profinite group $P$,  $\ol p \in P^k$  and $N \triangleleft P$ we will write $( P , \ol p ) /  N$ for the structure $(P/N, \seq{Np_i}_{i < k})$. 
 
 \begin{lemma} Suppose that $\Img(H ,\ol h) \sub \Img(G,  \ol g)$. Then $(G,\ol g) \rightsquigarrow (H, \ol h) $. \end{lemma}
 \begin{proof} Note that  a f.g.\ profinite group $P$ has only finitely many open  subgroups of each index. Let $P_n$ be the intersection of subgroups of index $\le n$. Then $P_n$ is a characteristic subgroup of finite index.  
 
Fix  $n$. By hypothesis  there is an open $N \triangleleft G$ such that  $(H, \ol h)/H_n \cong (G, \ol g)/N$. Then $G_n \le N$, and hence $(G, \ol g)/G_n \rightsquigarrow (H, \ol h)/H_n$.   Let $\Phi_n$ be the set of witnessing epimorphisms. 

If $\phi \in \Phi_{n+1}$ then $\phi(G_n/G_{n+1}) \le H_n / H_{n+1}$, so $\phi$ induces a map $\wt \phi \in \Phi_n$. By Koenig's Lemma applied to the tree where the $n$-th level is $\Phi_n$, there is an epimorphism  $\psi \colon (G, \ol g) \rightsquigarrow (H, \ol h)$ in the   inverse limit of the $\Phi_n$. 
 \end{proof} 
 
 \begin{proof}[Proof of the Theorem]
 We  extend the proof in  Segal~\cite[Thm.\ 4.2.3]{Segal:09} of the Jarden-Lubotzky theorem.  We show that $\Img(H ,\ol h) \sub \Img(G, \ol g)$. By symmetry, also $\Img (G, \ol h) \sub \Img(H ,\ol g)$. The lemma together with the Hopfian property of f.g.\ profinite groups now  implies that $(G,\ol g) \cong (H, \ol h) $.

 Suppose that $G$ is topologically generated by $d$ elements. Recall from  \cite{Nikolov.Segal:07} (also \cite[Ch.\ 4]{Segal:09}) that a group word $w$ is called \emph{$d$-locally finite} if the $d$-generated   free group in the variety of groups satisfying $w$ is finite. (As an  example, the word  $w= [x,y]z^m$ is $d$-locally finite for each $d$ because a f.g.\ abelian group of exponent $m$ is finite.) A main technical  result in  \cite{Nikolov.Segal:07} states  that for such a word, there is $f= f(w,d)$ such that $w(R) = R^{*f}_w$ for each $d$-generated finite group $R$, where $R^{*f}_w$ is the set of products of at most $f$ many $w$ values or their inverses. Note that $R^{*f}_w$ is definable by a $\exists$ formula depending only on $f$ and $w$.
 
 As in \cite[Thm.\ 4.2.3]{Segal:09} our hypothesis that $(G,\ol g) \equiv (H, \ol h) $ implies that $(G,\ol g)/G^{*f}_w \cong  (H, \ol h)/H^{*f}_w $. Checking the formulas reveals that only  $(G,\ol g) \equiv_{\Sigma_2} (H, \ol h) $ is needed.  
 
  Suppose $N \triangleleft H$ is open. We need to show that $(H, \ol h)/N$ is in $\Img(G, \ol g)$.
   Again as  in \cite[Thm.\ 4.2.3]{Segal:09} pick a $d$-locally finite word $w$ such that $H^{*f}_w \le N$. Then 
 
 \bc $(G, \ol g ) \rightsquigarrow (G,\ol g)/G^{*f}_w \cong  (H, \ol h)/H^{*f}_w   \rightsquigarrow  (H, \ol h)/N$,  \ec
 as required. 
 \end{proof}

     \section{Nies and Stephan: Some properties of \\ infinitely generated nilpotent-of-class-2 groups}

\subsection{Three properties of an elementary abelian group} Throughout fix a prime $p$. Let $G= F_p^{(\omega)}$ denote the elementary abelian $p$ group (i.e., vector space over  the field $F_p$ of infinite dimension). This group has the following apparently unrelated properties.

\begin{definition}  \mbox{} 
\begin{enumerate}  \item $G$ is word-automatic in the usual sense of Khoussainov and Nerode: finite automata can recognize the domain and the group operations.

\item $G$ is $\omega$-categorical: it is  up to isomorphism the only countable model of its theory.
\item $G$ is pseudofinite: if a sentence $\phi$ holds in $G$ then $\phi$ also holds in a finite group.  \end{enumerate}
\end{definition}
Note that a word automatic group is called finite automata presentable in \cite{Nies:DescribeGroups}  in order to avoid confounding the notion   with the better known notion of automatic groups due to Thurston.   The group $F_p^{(\omega)}$ is $\omega$-categorical because for each $n$ there are only finitely many $n$-types.  To show that $F_p^{(\omega)}$  is pseudofinite, note that its theory is axiomatized by stating that the group is infinite,  has exponent $p$ and is abelian. So we may assume that a sentence $\phi$ as above  is a finite conjunction of these axioms, and so  clearly has a finite model. 

\subsection{Nilpotent of class 2 groups}
Usually we will assume  that  $p\neq 2$.
We want to study the three properties above for groups that are very close to abelian. Recall that a group  $G$ is \emph{nilpotent  of class 2} if each commutator $[x,y]$ is in the centre $C(G)$. In other words, $G$ is an extension of a group $N$ by a group $Q$ that is  abelian,  and $N $ is   contained in $C(G)$ (i.e., $G$ is a central extension). It is not hard to show that such a group is entirely given by the abelian groups $N$, $Q$, and an alternating  bilinear map $L \colon \,  Q \times Q  \to N$. It determines a unique central extension $G$  such that $L(Nx,Ny) = [x,y]$ (note that $[x,y] $ only depends on the cosets of $x,y$).  In fact this is a special case of the well known fact that an extension of abelian group $N$ by $Q$  is given by an element of the second co-homology group $H^2(Q,N)$. 

We will study our properties for  three  examples of infinitely generated groups of exponent $p$.  We will always have $Q =F_p^{(\omega)}$,   and fix  a basis $\seq {x_i}\sN i$  of $Q$.

  Let $L_p$ be the free nilpotent-2 exponent $p$ group  of infinite rank. One  has $N\cong Q$;  let $y_{i,k}$ ($i< k$) be free generators of $N$. It is given by the function $\phi (x_i, x_k) = y_{i,k}$ if $i<k$.  Thus 
  
  \bc $L_p = \la x_i, y_{i,k} \mid x_i^p, [x_i, x_k] y_{i,k}^{-1} , [y_{i,k}, x_r] (i< k) \ra$ \ec
where the $y_{i,k}$ are actually redundant.

$G_p$ is the group where $Q = F_p^{(\omega)}$, $N = C_p$ (cyclic group of order $p$)  and   $\phi (x_i, x_k) = 1$ if $i<k$, $-1$ if $i > k$ (and of course $0$ if $i=k$). 

This group is   extra-special  in the sense of Higman and Hall. (This means that the centre is cyclic of order $p$, equals the derived subgroup, and the central quotient is an elementary abelian $p$-group.)

Felgner~\cite[Section 4]{Felgner:75} has proved that $G_p$ is $\omega$-categorical for $p\neq 2$. In his notation $G_p$ is $G(p, \le)$ where $\le$ is the ordering of $\omega$. This is the only countably infinite extra-special group of exponent $p$.

Felgner (p.\ 423) also provides a recursive  axiom system for the theory of $G_p$ ($p$ odd as before). 
One expresses that the group is of exponent $p$, that the centre is cyclic of order $p$, contains the derived subgroup, and that the derived subgroup is non-trivial (and hence equals the centre). Further one expresses that the central quotient is infinite,  using an infinite list of axioms. Note that for each odd $k$ there isa (unique) extraspecial $p$-group of exponent $p$ and   order $p^k$.  (For $k=3$ this is the upper triangular matrices over $F_p$, for larger $k$ one takes central products of these.) So, finitely many axioms can be satisfied in a finite model. Hence $G_p$ is pseudo-finite as already noted by Felgner.

The group $H_p$ has $N \cong Q$, with  free generators $z_k$ ($k >0$), and the function $\phi (x_i, x_k) = z_k$ if $i<k$:

\bc $H_p = \la x_i, z_{k} \mid x_i^p, [x_i, x_k] z_{k}^{-1} , [z_{k}, x_r] \,  (i< k, 0 \le r) \ra $. \ec

In a nilpotent group, each nontrivial normal subgroup intersects the centre non-trivially. This implies that every proper quotient of $G_p$ is abelian; in particular, $G_p$ is not residually finite. On the other hand $L_p$ and $H_p$ are residually $p$-groups. To see this, take an element $h \neq 1$, where $h$ depends on generators $x_0, \ldots, x_{n-1}$. Then $h \neq 1$ in the finite quotient $p$-group where all the generators $x_k$, $k \ge n$, become trivial. This means that both groups are embedded into their pro-$p$ completions.

\begin{fact} None of the groups $L_p, G_p, H_p$ is abelian by finite. \end{fact}
\begin{proof} Let $K$ be such a group and  suppose $M$ is a normal subgroup of finite index. There are $k< r <s$ such that $x_r x_k^{-1}\in M$ and $x_s x_k^{-1}\in M$. Then $[x_r x_k^{-1}, x_s x_k^{-1}] = [x_r, x_s] x_k^{-1}$, and this commutator is not $1$ in $K$. So $M$ is not abelian.  \end{proof}

\subsection{FA-presentability}
Any FA presentable structure has a decidable theory. Ershov~\cite{Ershov:72} showed  that a finitely generated nilpotent group has decidable first-order theory if and only if it is virtually abelian. 
 One the other hand, Nies and Semukhin~\cite[Thm.\ 4.2]{Nies.Semukhin:09} showed that  an abelian group that is an extension of finite index of an FA-presentable group is in itself FA-presentable. So a f.g.\ nilpotent group is FA-presentable iff it is abelian by finite.  Using Mal'cev coding, Nies and Thomas~\cite{Nies.Thomas:08} extended this by showing  that each finitely generated subgroup of an FA-presentable group is  abelian-by-finite. So  the natural setting  for finding interesting FA-presentable groups is among the  groups that are not finitely generated.
 
\begin{fact} $G_p$ is FA-presentable. \end{fact}

\begin{proof}  Each element has the form $c \cdot q$ where $c \in N$, $q\in Q$. 
 Note that   $q$ is given as a string $\alpha$ over the alphabet  of digits $0, \ldots, p-1$, and $c$ can be stored in the state. In the following $\alpha, \beta$ denote such strings, which are thought of as extended by $0$'s if necessary. Let $[\alpha] = \prod_i x_i^{\alpha_i}$.
 
  It is easy to verify that, with  arithmetic modulo $p$ and component-wise addition of strings,
 
 \[ [\alpha] [\beta] =  [\alpha + \beta] u^{\sum_{k>0}\alpha_k (\sum_{i< k} \beta_i}).\]
 This is because   one can  calculate $\prod_k x_k^{\alpha_k} \prod_i x_i^{\beta_i}$ by, for  decreasing positive $k$,    moving    $x_k^{\alpha_k}$ to the right past the $x_i^{\beta_i}$ for $i=0, \ldots, k-1$. Each such move creates a factor $u^{\alpha_k \beta_i}$.  
 
 A finite automaton processing $\alpha, \beta $ on two tracks can remember the current $\sum_{i< k} \beta_i$ in the state, and also carry out the calculation $\alpha_k (\sum_{i< k} \beta_i)$ and store the current exponent of $u$ in the state.
  \end{proof}

\begin{fact} $H_p$ is FA-presentable. \end{fact}
\begin{proof} An element $c\in N$ is of the form $ \prod_i z_i^{\gamma_i}$. An element $h$ of the group $H_p$ is represented by a pair of strings  $\alpha, \gamma$ such that $h= [\alpha]  \prod_i z_i^{\gamma_i}$. These strings are written on two tracks, $\alpha$ on top of $\gamma$.

Note that by a similar argument as above, in $H_p$ we have
 \[ [\alpha] [\beta] =  [\alpha + \beta] \prod_{k>0} z_k^{\alpha_k (\sum_{i< k} \beta_i)}.\]
 By the argument above the FA can store the value ${\alpha_k (\sum_{i< k} \beta_i)}$ for increasing $k$. This shows that the group operation can be checked by finite automaton.  
\end{proof}
\begin{conjecture} Show that $L_p$ is not FA presentable.\end{conjecture}

If only finitely many things do not commute with the others then
it is not isomorphic to the group which we constructed. So it is
a further example of such a group. The reason is quite easy: If
$x_1,..,x_k$  do not commute with each other and $y_1,y_2,...$  commute
all with each other but not with $x_1,..,x_k$  then every basis has
two things which commute.

\begin{fact} The pro-$p$ completion $\hat H_p$ of $H_p$ is B\"uchi presentable. \end{fact}

\begin{proof} Let $\hat Q$ be the pro-$p$ completion of $Q$, which is $C_p^\omega$ (cartesian power) viewed as a compact group. $\hat H_p$ is a central extension of $\hat Q$ by itself where the first $\hat Q$ has  the topological generators $z_k$, the second has the  $x_i$, and we have $[x_i, x_k] = z_k$ for $i< k$, as before. 

An element $h$ of the group $\hat H_p$ is represented by a pair of infinite strings  $\alpha, \gamma$ such that $h= [\alpha]  \prod_i z_i^{\gamma_i}$. 
So we can use the same FA as above, but  working on infinite strings, to check the group operation.  \end{proof}

%

%

         \part{Computability theory and set theory} 
         
            \section{Greenberg and Nies: cardinal characteristics and computability}

 Nicholas Rupprecht,  in his  2010 PhD thesis   supervised by  Blass \cite{Rupprecht:thesis}, studied the connection of cardinal characteristics of the continuum in set theory and   highness properties from computability theory. He developed this connection systematically as coming from a single source. Somewhat unfortunately, the thesis is  somewhat  technically written and  the notation is hard to access, which may be  a reason for the fact  that no immediate reaction by either of the two communities  to this unexpected connection resulted. 
 
  The   cardinal characteristics given by null sets correspond to highness properties defined in terms of algorithmic randomness. For instance, $\text{non}(\+ N)$, the least size of a non-null set of reals, corresponds to the strength of an oracle that computes a Schnorr test containing all recursive bit sequences. Rupprecht published a single   paper mainly on the computability theoretic impact of that notion~\cite{Rupprecht:10}.  Rupprecht called this property ``Schnorr covering", which is confusing because it does \emph{not} correspond to  the characteristic $\text{cov}(\+ N)$ (the least size of a class of null sets with union $\mathbb R$); later on, Brendle et al.\  suggested the term  ``weakly Schnorr engulfing"    \cite{Brendle.Brooke.ea:14}.   We note that Rupprecht had attended the NSF FRG-funded  school on computability and  randomness at Gainesville in 2008,  where Nies gave a tutorial on randomness; this may have prompted him to make the connection between measure theoretic cardinal invariants and randomness.   He    started a career as an investor shortly after his defence.  According to LinkedIn his last job was as  the Global Director of Operations at Virtu Financial.  Blass mentioned he didn't work there any longer. Apparently he  has disappeared from view.
 
A quote  from Brendle et al.\  \cite{Brendle.Brooke.ea:14}: the analogy  occurred  implicitly in the work of Terwijn and Zambella \cite{Terwijn.Zambella:01}, who showed that being  low for Schnorr tests is equivalent to  being computably traceable. (These are  lowness properties, saying the oracle is close to being computable; we obtain  highness properties by taking   complements.)  This is the computability theoretic analog of a result by Bartoszy\'nski~\cite{Bartoszynski:84} that the cofinality of the null sets (how many null sets does one need to cover all null sets?) equals the domination number for traceability.
  (Curiously, Terwijn and Zambella alluded to some connections with  set theory in their work. However, it was not Bartoszy\'nski's work, but rather   work on rapid filters by Raisonnier~\cite{Raisonnier:84}.) Their proof bears striking similarity to Bartoszy\'nski's; for instance, both proofs use measure-theoretic calculations involving independence.  See also the book references \cite[2.3.9]{Bartoszynski.Judah:book} and \cite[8.3.3]{Nies:book}. (End quote.)

             \section{Lempp, Miller, Nies and Soskova:   Analogs in computability \\ of combinatorial cardinal characteristics}

\n \emph{Brief summary.} Our basic objects are infinite sets of natural numbers. In set theory,   the MAD number is the least cardinality of a maximal almost disjoint class of sets of natural numbers. The ultrafilter number is the least size of an ultrafilter base.

We study computability theoretic analogs of these cardinals. In our approach, all the basic objects  will be infinite computable sets. A class of such basic objects is  encoded as the set of  ``columns" of a set R, which allows us to study the  Turing complexity of the class.

We show that each non-low oracle computes a MAD class R, give a finitary construction of a c.e. MAD set (compatible with permitting), and on the other hand show that a 1-generic below the halting problem does not compute a MAD class.  We also provide some initial results on ultrafilter bases in this setting.

\medskip
The original impetus for this work  resulted from some interactions of Nies with Brendle, Greenberg and others  at the August meeting on set theory of the reals  at Casa Matem\'atica Oaxaca in early August, which was  organised by Hrusak, Brendle, and Steprans.  The computability results reported below are due to Lempp, Miller, Nies and Soskova (in prep.), based on joint work that happened during Nies' visit in Madison end of August, 2 weeks after  the end of the CMO meeting. Downey's visit in Madison in September further  advanced it.

\subsection{Set theory.} Let $\frb$  be the unbounding number (the least size of a class of functions on $\omega$ that is not dominated by a single function). 

Several cardinal characteristics are based on cardinals  of subclasses of $[\omega]^\omega/{=^*}$, i.e. the infinite subsets of $\omega$  under almost equality.  One of them is    called the almost disjointness number,  denoted $\mathfrak a$. This  is  the minimal size of a maximal almost disjoint (MAD) family of subsets of $\omega$. 
The following is  well known.
\begin{fact}  \label{fa: ba} $\frb \le \mathfrak a$. \end{fact}

\begin{proof}  Suppose that the infinite cardinal $\kappa$ is less than $\frb$. Let $\+ V$ be an almost disjoint family of size  $\kappa$. We show that $\+ V$ is not MAD. Let $\seq {r_n}$, $n \in \omega$,  be a sequence of pairwise distinct elements of $\+ V$. Given $y \in \+ V$ let \bc $f_y(n) = \max  (y \cap r_n)$ \ec
with the convention that $\max \ES = 0$. Let $g$ be an increasing  function such that $g(n) \in r_n$ for each $n$  and $\forall y \in \+ V \forall^\infty n  \, [ g(n) > f_y(n)]$.  It is clear that the family $\+ V \cup \{\range (g)\}$ is almost disjoint.  \end{proof}

Other  cardinal characteristics are based on properties   of subclasses of  the infinite sets modulo almost equality: 
\bi \item  the ultrafilter number $\fru$ (the least size of a set  with upward closure  a free ultrafilter on $\omega$), \item the  tower number $\mathfrak t$ (the minimum size of a   linearly ordered subset that can't be extended by putting a new element below all given elements), \item the  independence number $\mathfrak i$  \ei  and several others. See e.g.\ \cite{Blass:10} or the  talk by Diana Montoya at the CIRM meeting on descriptive set theory in June 2018. This talk  contained a Hasse diagram of these cardinals with ZFC inequalities, reproduced here with permission:

\bc \includegraphics[width=2.5in]{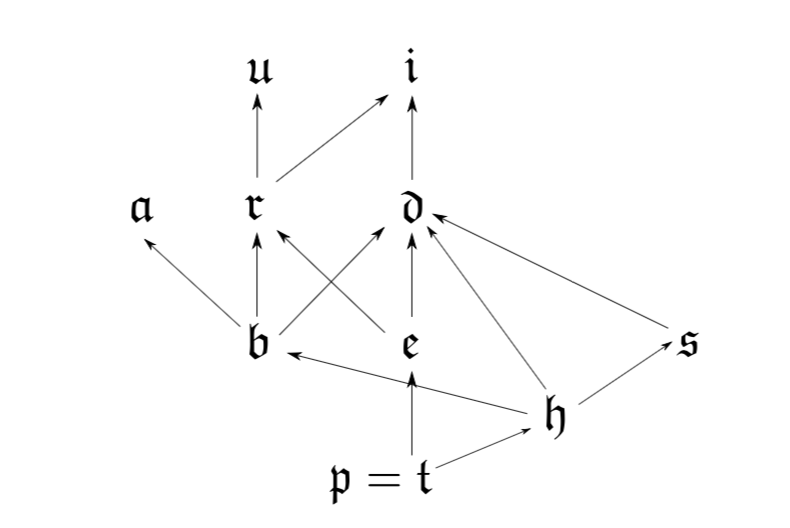} \ec

Note that $\mathfrak r$ and $\mathfrak s$ are the unreaping and splitting numbers, respectively. They are given by relations, and their analogs in computability  have been studied \cite{Brendle.Brooke.ea:14}. Furthermore,  $\mathfrak e$ is the escaping number due to Brendle and Shelah. Its analog has   recently been considered in computability theory by Ivan Ongay Valverde and  Paul Tveite, two PhD students from Madison. Finally,   $\mathfrak h$ is the distributivity number.   For another diagram see Soukup~\cite{Soukup:18}.

\subsection{Analogous mass problems   in computability} For a set $F$,   \bc $F^{[n]}$   denotes   the column $\{ x \colon \la x, n \ra \in F\}$.  \ec We will also denote this by $F_n$ if no confusion is possible. 

 A systematic way to translate these characteristics into highness properties  of oracles  is  as follows. 
We consider  subclasses of $\text{Rec}/{=^*} \setminus \{0\}$, identified with classes of   infinite recursive sets in the context of  almost inclusion.    Such a subclass $\+ C$  is encoded by a set $F$ such that \bc $\+ C = \+ C_F =   \{F^{[n]} \colon n \in \NN\}$. \ec We view these classes as mass problems. We   compare them via Muchnik reducibility $\le_W$ and the stronger, uniform Medvedev reducibility $\le_S$.

\subsection{The mass problems $\+ A$ and $\+ T$}
We say that (the class $\+ C_F$ described by) $F \sub \NN$ is a \emph{almost disjoint, AD in brief,} if each $F_n$ is infinite, and  $F_n \cap F_k $ is finite for  each $n\neq k$.

\begin{definition} The mass problem  $\+ A$  corresponding to the almost disjointness number $\fra$ is the class of   sets $F$  such that $\+ C_F$ is  maximal almost disjoint (MAD)  at the recursive sets. Namely, $\+ C_F$ is AD,  and   for each infinite recursive set $R$ there is $n$ such that $R\cap F_n $ is infinite. \end{definition}
We need to get this out of the way:
\begin{fact}\label{fa:silly} No MAD set $F$ is computable. \end{fact}
\begin{proof}
Suppose $F$ is AD and computable. Let $r_{-1} =0$, and $r_{n}$ be the least number $r > r_{n-1}$  such that $ r \in F_n - \bigcup_{i < n} F_i$. Then the computable set $R = \{r_0, r_1, \ldots\}$ shows that $F$ is not MAD. 
\end{proof}
\begin{definition} \label{def:tower}  We say that $G \sub \NN$ is a \emph{tower} (or $\+ C_G$ is a    tower) if for each $n$ we have  $G_{n+1} \sub^* G_n$ and $G_n - G_{n+1}$ is infinite. 

By a function  \emph{associated} to $G$  we mean an  increasing   function $p$ satisfying   the conditions $p(n ) \in \bigcap_{i\le  n} G_i$.
 \end{definition}

\begin{definition} The mass problem  $\+ T$  corresponding to the tower    number $\mathfrak t$ is the class of   sets $G$  such that $\+ C_G$ is a    tower that is maximal in the   recursive sets. Namely,   for each infinite recursive set $R$ there is $n$ such that $R- G_n $ is infinite. 

  \end{definition} 

\begin{fact} No tower is c.e. \end{fact}
Otherwise there is a computable associated function $p$. The range of  $p$ would extend the tower.

\begin{fact} \label{fa:ATA} $\+ A \le_S \+ T \le_S \+ A$. \end{fact} 
\begin{proof} To check that $\+ A \le_S \+ T$, given a set $G$ let $\mathrm{Diff}(G) $ be the set $D$ such that $D_n = G_n - G_{n+1}$. Clearly the operator $\mathrm{Diff} $ can be seen as a Turing functional. 
If $G$ is a maximal    tower then $D=\mathrm{Diff}(G) $ is MAD. For, if $R$ is infinite recursive then $R-G_n$ is infinite for some $n$,  and hence $R\cap D_i$ is infinite for some $i<n$.

For $\+ T \le_S \+ A$, given a set $F$ let $G =\mathrm{Cp}(F) $ be the set   such that  \[ x \in G_n \lra \fa i < n \, [ x \not \in F_n].\]
Again $\mathrm {Cp}$ is a Turing functional. If $F$ is AD then $G$ is a    tower, and if $F$ is MAD then $G$ is a maximal    tower.
\end{proof}

Recall that a characteristic index for a set $M$ is an $e$ such that $\chi_M= \phi_e$.

\begin{prop}  \label{pr:halt} Suppose $F \in \+ A$. While each $F_n$ is computable, $\Halt $ is not able to compute, from input  $n$,  a characteristic index for $F_n$. 
\end{prop}
\begin{proof} Assume the contrary. Then there is a computable function $f$ such that  $\phi_{\lim_s f(n,s)}$ is the characteristic function of $F_n$. 

Let $\hat F$ be defined as follows. Given $n, x$ compute the least $s> x$ such that $\phi_{f(n,s),s}(x) \DA$. If the value is not $0$ put $x$ into $\hat F_n$. 

Clearly $\hat F$ is computable and $F_n = ^* \hat F_n$ for each $n$. So $\hat F$ is MAD, contrary to Fact~\ref{fa:silly}. \end{proof}

 This research started  at the CMO by looking at   {an analog of $\frb \le \fra$.}  This isn't quite the right approach as we will see. At least,  it is misleading. 
 
 The analog of $\frb$ is the  property of an oracle $C$ to be  \emph{high}, namely  $\ES'' \le_T C'$; for detail  on this see~\cite{Rupprecht:10,Brendle.Brooke.ea:14}.  (To see this as a mass problem,  take the functions eventually dominating all computable functions,   instead.) We want  to   show that an analog of the relationship $\frb \le \fra$  holds: every high oracle computes a set in $\+ A$.  The inversion  that occurs here is part of the generally accepted setup of analogy  between the two areas of logic.
One can show that each high oracle  $C$  computes a set  in $\+ A$.  The proof is a straightforward application of a dominating function computed by $C$. 

Lempp, Miller, Nies and Soskova (in prep.) have strengthened the result:  each non-low oracle computes a set  in $\+ T$, and hence in $\+ A$.   The result is uniform in the sense of mass-problems. Let NonLow denote the class of oracles $X$ such that $X' \not \le_T \ES'$. 

\begin{thm} \label{thm: NonLow} $\+ T \le_S \mathrm{NonLow}$.  \end{thm}
\begin{proof}
In the following $x,y,z$ denote binary strings; we identify $x$    with the number $1x$  via the binary expansion. Define a Turing functional $\Phi$ for the Medvedev reduction, as follows. $\Phi^Z = G$, where for each $n$ 
 \[G_n= \{x \colon \, |x| = s \ge  n \land Z'_s \uhr n=  x \uhr n\}.\]
It is clear that for each $n$ we have  $G_{n+1} \sub^* G_n$ and $G_n - G_{n+1}$ is infinite. Also, for each $n$, for large $s$ the string $Z'_s \uhr n$ settles, so $G_n$ is computable.

Suppose now that $R$ is an infinite set such that $R \sub^* G_n$ for each $n$. Then $Z'(k)= \lim_{x \in G, |x| > k } x(k)$, and hence $Z'\le_T R'$. So if $Z \in \mathrm{NonLow}$ then $R$ cannot be computable, and hence $\Phi^Z \in \+ T$. 
\end{proof}

\subsection{$\Delta^0_2$ and $1$-generic sets, totally low sets, and not computing a MAD}

We provide   a type of sets which doesn't compute a MAD set. Joe Miller and Mariya Soskova showed (during the weekend after Andre had left) that if $L$ is $\Delta^0_2$ and $1$-generic,  then  $L$ does not compute a MAD family. Downey's visit in Madison before the Lempp `60 conference brought another interesting lowness property into the game: totally low oracles. This property, derived from Joe's and Mariya's proof, is in between    $1$-generic $\Delta^0_2$, and not computing a MAD. No separation is known at present.

Note that for a $\DII$  oracle $L$, $\emptyset''$ can compute from $e$ an index for $\Phi_e^L$, provided that it's computable. To be totally low means that $\Halt $ suffices.
\begin{definition} We call an oracle $L$ \emph{totally low} if $\Halt$  can compute from $e$ an  index for $\Phi_e^L $ whenever $\Phi_e^L$ is computable. In other words, there is a functional $\Psi$  a such that 
 \bc $\Phi_e^L$ computable $\RA$ $\Psi(\Halt;e)$ is an index for it. \ec \end{definition}
 \n No assumption is made on the  convergence  of $\Psi(\Halt;e)$ in case $\Phi_e^L$ is not a computable function. 
 A total function is computable iff its   graph is.  A moment's thought using this fact   shows   that we may restrict ourselves to $e$ such that $\Phi_e^L$ is $\{0,1\}$-valued.  Clearly, being totally low is closed downward under $\le_T$.
 
  By  Prop~\ref{pr:halt}, a totally low oracle $D$ does not compute a MAD set.  In particular by Theorem~\ref{thm: NonLow}, $D$ is low  (as the name suggests)
\begin{thm}
Suppose $L$ is $\Delta^0_2$ and $1$-generic. Then  $L$ is totally low.
\end{thm}
\begin{proof}
Suppose   $F = \Phi_e^L$ and $F$ is a computable set. 
Let \[S_e = \set{\sigma}{ (\exists \tau_1\succeq\sigma)(\exists \tau_2\succeq \sigma)(\exists p)[\Phi_e^{\tau_1}(p)\neq \Phi_e^{\tau_2}(p)]}.\]
 Suppose that $S_e$ is dense along $L$.  We claim that  the set
  \[C_e = \set{\tau}{ (\exists p)[\Phi_e^{\tau}(p)\neq F(p)]}\]
   is also dense along $L$: i.e. for every $k$ there is some $\tau\succeq L\uh k$ such that $\tau\in C_e$. Indeed, let $\sigma\succeq L\uh k$ be a member of $S_e$ and let $\tau_1$ and $\tau_2$ and $p$ witness that. Let $\tau_i$ where $i = 1$ or $2$ be such that $ \Phi_e^{\tau_i}(p)\neq F(p)$. Then $\tau_i\succeq L\uh k$ is in $C_e$. The set $C_e$ is c.e.\ and hence $L$ meets $C_e$, contradicting our assumption that $F = \Phi_e^L$. 
 
 It follows that $S_e$ is not dense along $L$.  In other words,  there is some least $k_e$ such that  there is no splitting of $\Phi_e$ above $L\uh k$. On input ~$e$ the oracle $\emptyset'$ can compute $k_e$ and $L\uh k_e$. This allows $\emptyset '$ to find an index for $F$, given by the following  procedure. To compute $F(p)$, find the least $\tau\succeq L\uh k_e$ such that $\Phi_e^{\tau}(p)\downarrow$ (in $|\tau|$ many steps). Such a $\tau$ exists because $\Phi_e^L(p)\downarrow$. By our choice of $k_e$ it follows that $\Phi_e^{\tau}(p) = \Phi_e^L(p) = F(m)$. 
\end{proof}
\n 
By EX one denotes the class of computable  functions $f$ so that an index can  be learned in the limit: there is a machine $M$ taking values $f(0), f(1), \ldots $ such that $\lim_s M(f(0), \ldots, f(s)$ exists and is an index for $M$. To relativize  this to oracle $X$, allow $M$ access to $X$ (but the functions are still computable).
Recall that $A$ is low for EX learning if $EX^A= EX$. Slaman and Solovoy \cite{Slaman.Solovay:91} showed that \bc $A$ is low for EX-learning $\LR$ $A$ is $\DII $ and 1-generic. \ec 

We summarize the  known implications of lowness notions.  

\bc   $\DII $ and 1-generic $\RA$ Totally low $\RA$  
 computes no MAD $\RA$ low. \ec

The last arrow doesn't reverse by what follows; the others might.

\subsection{C.e.\ MAD sets}
It may come as a surprise that a set in $\+ A$ can be c.e., and in fact can be  built by  a priority construction with finitary requirements, akin to Post's construction of a simple set.  Such a  construction  is compatible with permitting.  The result to follow  is part of the joint work  mentioned above.
\begin{thm} For each incomputable c.e.\ set $A$, there is a MAD c.e.\ set $R \le_T A$. \end{thm}

\begin{proof} Let $\seq {V_e} \sN e$ be a u.c.e.\ sequence of sets such that $V_{2e} = W_e$ and $V_{2e+1} = \NN$ for each $e$. We will build an auxiliary   c.e.\ set $S \le_T A$ and let  the c.e.\ set $R\le_T A $ be defined   by $R^{[e]} = S^{[2e]} \cup S^{[2e+1]}$. The role of the  $V_{2e+1}$ is to make the sets $S^{[2e+1]}$, and hence the  $R^{[e]} $ infinite. The construction also ensures that $S$, and hence $R$, is AD, and that $\bigcup_n S^{[n]}$ is coinfinite. 

We will write $S_e$ for $S^{[e]}$.
We provide a stage-by-stage construction to meet   the requirements $P_n$,  $  n= \langle e, k \rangle$ given by

$$P_n  \colon V_e - \bigcup_{i<n} S_i \text { infinite } \Rightarrow |S_e \cap V_e | \ge k.$$  
At stage $s$ we  say that $P_n$ is satisfied  if $|S_{e,s} \cap V_{e,s} | \ge k$. 

\noindent \emph{Construction.} 

\noindent \emph{Stage $s>0$}. For each $n< s$ such that $P_n$ is not satisfied, $  n= \langle e, k \rangle$,  if there is $x \in V_{e,s} -  \bigcup_{i<n} S_{i,s}$ such that    $x >  \max (S_{e,s-1})$, $x \ge 2n$  and $A_s \uhr   x \neq A_{s-1} \uhr x$, then put   $\langle x, e \rangle$ into $S$ (i.e., put $x $ into $S_e$). 

\noindent \emph{Verification.}
Each $S_e$ is computable,  being enumerated in an increasing fashion. 

Each $P_n$ is active at most once.  This ensures that $\bigcup_e S_e$ is coinfinite: for each $N$, if $x< 2N$ enters this union then   this is due to the action of a  requirement $P_n$ with $n \le N$, so there are at most $N$ many such $x$. 

To see that  a requirement $P_n$, $  n= \langle e, k \rangle$,  is met, suppose that its hypothesis holds. Then there are potentially infinitely many candidates $x$ that can go into $S_e$. Since $A$  is incomputable, one of them will be permitted.

 Now, by the choice of $V_{2e+1}$, each $S_{2e+1}$, and hence each $R_e$ is infinite.
By construction, for $e< m$ we have $|S_e \cap S_m| \le m$. So the family described by $S$ and therefore also the one described by  $R$ is almost disjoint.

To show $R$ is MAD, it suffices to verify  that 
if $V_e$ is infinite then $V_e \cap R_p$ is infinite for some $p$. If all the $P_{e,k}$ are satisfied during the construction then we let $p=e$. Otherwise let $k$ be least such that $P_{n}$ is never  satisfied where $n= \langle e,k \rangle$. Then its hypothesis fails, so  $V_e \sub^* \bigcup_{i<n} S_i $. 
\end{proof}
By the argument for the reduction  $ \+ T \le_S \+ A$  in Fact~\ref{fa:ATA}, we obtain 
\begin{cor} For each incomputable c.e.\ set $A$, there is a co-c.e.\ set $G \le_T A$ such that $G$ describes a maximal tower, i.e. $G \in \+ T$. \end{cor}
Namely,   $G =\mathrm{Cp}(A) $ i.e., \[ x \in G_n \lra \fa i < n \, [ x \not \in A_n].\]

\begin{fact} \label{fact:diag} From a  co-c.e.\ tower $G$ one can uniformly compute an infinite   co-c.e.\ set $B $ such that $B \sub^* G_n$ for each $n$. Moreover, the index for $B$ is also obtained uniformly. \end{fact}
\begin{proof} We may assume that $G_n \supseteq G_{n+1}$ for each $n$. Let $\gamma_{n,s}$ be the $n$-th element of $G_{n,s}$. Clearly $\gamma_{n,s}$ is monotonic in $s$ and strictly monotonic in $n$. Let $\gamma_n = \lim_s \gamma_{n,s}$. The set  $B= \{\gamma_n \colon \, n \in \omega \}$ is as required. \end{proof}

Since a totally low set computes  no MAD set, we obtain:
\begin{cor} No incomputable,  c.e.\ set $L$ is   totally low. \end{cor} 

Downey and Nies  in October (during Nies' Wellington visit where Nies gave a seminar talk about this stuff) have given a direct proof of this. Assume for contradiction that $L$ is incomputable, c.e., and totally low. Define Turing functionals $\Gamma_e$, uniformly in $e$. By the recursion theorem, we  know in advance a computable function $p$ such  $\Gamma^X_e(n) = \Phi_{p(e)}^X(n)$ for each $X, n$. 

We meet  for each $e$ the requirement \bc  $P_e \colon \Gamma_e^L$ is computable but $\Phi_e^{\Halt}(p(e))$ is not an index for it. \ec
 Write $\alpha (e,s)$ for $\Phi_e^{\Halt}(p(e))[s]$.  
 
 \n {\it  Construction.}  Assign an increasing sequence of followers $x_e \in \omega^{[e]}$ to $P_e$. Initially we set $\Gamma_e^L(x_e)= 0$ with use $x_e$.  Once $\phi_{\alpha(e,s)}(x_e)=0$, we call $x_e$ \emph{certified} for $e$. If $L$ at a certain stage $t$ permits a certified $x_e$ that is  larger than any follower that is fulfilled, we change the value $\Gamma_e^L(x_e)$ to $1$, and call $x_e$ \emph{fulfilled}.

 \n {\it Verification.}  Since followers are enumerated in an  increasing fashion,  $\Gamma^L_e$ is computable. 
 There is a  final value $\lim_s\alpha(e,s)$, which hence is an index for $\Gamma^L_e$. After $\alpha(e,s)$ has stabilized to a value $v$, infinitely many followers get certified. So $L$ permits at a stage $t$ a certified  follower for $P_e$. This follower gets fulfilled  at  stage $t$, which causes the  index $v$ to become incorrect.  Contradiction.

\subsection{The mass problem  $\+ U$ corresponding to the ultrafilter  number,  and a strict reduction $ \+ T <_S \+ U$}

\begin{definition} Let  $\+ U$  be  the class of  sets $F $  such that $F$ is tower  as in Def.\ \ref{def:tower},  and $\+ C_F$ is an ultrafilter base within the recursive sets: for each   recursive set $R$ there is  $n$ such that  $F_n \sub^* \overline R$ or $F_n \sub^*  R$. \end{definition}
 Since $\+ U \sub \+ T$ we trivially have $\+ T  \le_S \+ U$  via the identity reduction.
Since our   properties of (codes for) classes are arithmetical, each of the  mass problems   introduced  contains an element computable from  $\ES^{(n)}$ for sufficiently large $n$. For instance,  $\ES''$ computes an oracle in  $ \+ U$. To see this,  take any $r$-cohesive set $C$. By definition of $r$-cohesiveness,   the recursive sets $R$ such that $  C \subseteq^* R$ is cofinite form an ultrafilter. If e.g.\ $\overline C$ is $r$-maximal, we can write  this ultrafilter as $\+ C_F$ for some $F \le_T \ES''$. 

For the next fact we  follow the upcoming work mentioned above due to Lempp, Miller, Nies and Soskova.  In the cardinal setting  we have $\mathfrak t \le \mathfrak u$, which suggests that each maximal tower should compute an ultrafilter base. But this is not true by the following.
 
\begin{prop} \label{prop:CDno} No ultrafilter base  $R$ is computably dominated. \end{prop} 

\begin{proof} Let $p(n)$ be  an associated function as in Def.\ \ref{def:tower}, i.e. an increasing function $p$ such that   $p(n) \in \bigcap_{i\le n} R_i$. Then $ p \le_T R$. Assume that there is a computable function $f \ge p$. The conditions $n_0=1$ and $n_{k+1} = f(n_k)$ define  a computable sequence. So the set  
\[E= \bigcup_i [n_{2i}, n_{2i+1}) \]  is computable.
Clearly $R_n \not \sub^* E$ and $R_n \not \sub^* \ol E$ for each $n$.  So $R$ is not an ultrafilter base.  \end{proof} 

On the other hand,   each non-low set computes  a set in $\+ A \equiv_S \+ T$ as shown above, and of course a computably dominated oracle can be non-low.

\subsection{Co-c.e.\   ultrafilter bases}

Recall that no tower, and hence no ultrafilter base, is c.e. \begin{fact} Every co-c.e.\  ultrafilter base $G$ is high. \end{fact}
\begin{proof} 
Let $B$ be a co-c.e.\ set as obtained in Fact~\ref{fact:diag}. Clearly $B \sub^* R$ or $B \sub^* \ol R$ for each computable $R$. Hence $\ol B$ is $r$-maximal, and therefore    high. Since $B \leT G$, $G$ is also high.
\end{proof}

Let us say for the moment that a tower  $F$ is a \emph{ultrafilter base for c.e.\ sets}  if for each \emph{computably enumerable} set $R$ we have  $F_n \sub^* \overline R$ or $F_n \sub^*  R$.  By Fact~\ref{fact:diag} there is no co-c.e.\   ultrafilter base for c.e.\ sets.  To  build a co-c.e.\ ultrafilter base we need to make use of the fact that we are given c.e.\ index for a computable  set and also one for its complement.

\begin{thm} There is   a co-c.e.\   ultrafilter base $F$. \end{thm}


%

%
%
%
%
 \begin{proof}   
 We build the co-c.e.\ tower $F$ by providing uniformly co-c.e.\ sets $F_e$, $e \in \NN$ that form  a descending sequence: $F_e \supset_\infty F_{e+1}$.  If we remove $x$ from $F_e$ at a stage $s$ we also remove it from all $F_i$ for $i >e$, without further mention. 
 
 Let $\seq{(V_{e,0}, V_{e,1})}$ be an effective listing of the  pairs of disjoint c.e.\ sets.
The construction  will ensure that  the  following requirements are met. 
\bc $\+ M_e \colon F_e \setminus F_{e+1}$ is infinite. \ec
 
 \bc $\+ P_e \colon V_{e,0} \cup V_{e,1} = \NN \RA F_{e+1} \sub^* V_{e,0} \vee F_{e+1} \sub^* V_{e,1}$. \ec
 This suffices to establish that  $F$ is an ultrafilter base.

 The tree of strategies is $T=\{0,1,2\}^{< \infty}$. Each string $\aaa \in T$ of length $e$ is associated   with $M_e$ and also $P_e$. We write $\aaa \colon M_e$ and $\aaa \colon P_e$ to indicate that we view $\aaa$ as a strategy of the respective type.

 \medskip
 \n \emph{Streaming.} For each string  $\alpha\in T$, $|\aaa| = e$,  at each stage of the construction we have a  set $S_\alpha$, thought of as a stream of numbers used by $\aaa$.  Each time $\alpha$ is initialized, $S_\alpha$ is made empty and its content removed from $F_{e+1}$. Also, 
  $S_\alpha$ is enlarged only at stages at which $\alpha$ appears to be on
the true path. We will verify the following properties.
\begin{enumerate}
\item $S_\emptyset = \omega$;
\item if $\alpha$ is not the empty node then $S_\alpha$ is a subset of
$S_{\alpha^-}$ (where $\alpha^-$ is the immediate predecessor of $\alpha$);
\item At every stage  $S_\gamma \cap S_\beta = \ES$ for incomparable strings $\gamma$
and $\beta$;
\item at the time a number $x$ first enters $S_\alpha$, $x$ is   in $F_{e+1}$;
 and
\item if $\alpha$ is along the true path of the construction then $S_\alpha$
is an infinite computable set.
\end{enumerate}
   Thus $S_\alpha$ can be thought of as a set d.~c.~e.\ uniformly in $\alpha$; $S_\alpha$ is finite if $\alpha$ is to the left of the true path of the
construction; an infinite computable set if $S_\alpha$ is along the true path;
and empty if $\alpha$ is to the right of the true path.
 
 \medskip
  \n  \emph{Construction.} 
  
  \n Stage $0$. Let $\delta_0$ be the empty string.  Let $F_{e}= \NN$ for each $e$.  
 
  \n  \emph{Stage} $s+1$.     Let $S_{\ES,s}= [0,s)$.  
  
   \n  \emph{Substage $e<s$.} We suppose that  $\alpha= \delta_{s+1}\uhr e$ has been defined.   
  
  \medskip 
 \n Strategy  $\aaa \colon M_e$  removes  every other element of $S_\aaa$ from $F_{e+1}$.  Let $S'_\aaa$ denote the set of remaining numbers.  More precisely, if we have $S_\aaa = \{r_0< \ldots <r_k\}$,  then  $S'_\aaa$ contains the numbers of the form $r_{2i}$ while the numbers of the form $r_{2i+1}$ are removed from $F_{e+1}$.

   \medskip 
 \n Strategy  $\aaa \colon P_e$ picks the first applicable case below.
  
\n 1. \emph{Each    reserved number of $\aaa$ has been  processed.}

\n  If there is a  number $x$  from $S'_\aaa$ greater than $\aaa$'s last reserved number (if any) and greater than $s_0$, pick $x$ least and \emph{reserve} it.  Initialize $\aaa\ape 2$. Let $\aaa\ape 2$ be eligible to act next. 
  
  \n If (1) doesn't apply then  $\aaa$ has a reserved,  unprocessed number~$x$.
  
  \smallskip 
\n 2.    $[0,x] \sub V_{e,0} \cup V_{e,1}$  and $x \in V_{e,0}$.  
 
 \n Let $s_0$ be greatest stage $<s$ at which $\alpha$ was initialized. Add  $x$ to $S_{\aaa 0}$  and remove from $F_{e+1}$  all numbers in the interval $(s_0,x)$ which are not in
$S_{\aaa 0}$. Declare that    $\aaa$ has \emph{processed} $x$.
 Let ${\aaa \ape  0}$ be eligible to act next. 
 
   \smallskip
   
\n 3. $[0,x] \sub V_{e,0} \cup V_{e,1}$  and $x \in V_{e,1}$. 

\n  Let $s_0$ be greatest stage $<s$ at which $\alpha$ was initialized or $\aaa\ape 0$ was eligible to act. Add  $x$ to $S_{\aaa 1}$  and remove from $F_{e+1}$  all numbers in the interval $(s_0,x)$ which are not in
$S_{\aaa 1}$.  Declare that    $\aaa$ has \emph{processed} $x$.
 Let ${\aaa \ape 1}$ be eligible to act next.  
 
   \smallskip
   
\n 4. \emph{Otherwise.}

\n   Let $s_0$ be greatest stage $<s$ at which $\alpha$ was initialized or $\aaa  \ape 2$ was eligible to act. 
Let $S_{\aaa 2}= S'_\aaa \cap (s_0, s)$.  Let $\aaa\ape 2$ be eligible to act next. 

  \smallskip
 
 We   define $\delta_{s+1}(e)=i$ where  $\aaa \ape  i$, $0 \le i \le 2$, has been declared   eligible to act next.

   \medskip

  \n   \emph{Verification.} By construction and our convention above,  $F_e \supseteq F_{e+1}$ for each $e$, and $F$ is co-c.e.

  Let $g \in \tp \omega$ denote the true path, namely the  leftmost path in $\{0,1,2\}^\omega$ such that $\forall e \ex^\infty s \, [ g\uhr e \preceq \delta_s]$. In the following,  given $e$ let $\aaa = g \uhr e$, and let $s_\aaa$ be the largest stage $s$ such that $\aaa$ is initialized at stage $s$. We  verify a number of claims.

  \begin{claim}  The ``streaming" properties (1)-(5) hold. \end{claim}
   \n (1,2) hold by construction.
   
\n     (3)  Assume this fails for incomparable $\gamma, \beta$, so $x \in S_\gamma \cap S_\beta$ at stage $s$. We may as well assume that $\gamma= \aaa \ape i$ and $\beta = \aaa \ape k$ where $i<k$. By construction $k\le 1$ is not possible, so $k=2$. Since $x \in S_\aaa\ape i$ adn $i \le 1$, $x$ was reserved  by $\aaa$ at some stage $t\le s$. 
So $x$  can never  go into $S_{\aaa \ape 2}$ by the initialization of $\aaa\ape 2$ when $x$ was reserved.

  \n  (4) is  true by construction.

\n   (5) holds by definition of the true path, and because  $S_\aaa$ is enumerated in increasing fashion at stages  $\ge s_\aaa$.
  
  \begin{claim} \label{cl:3}
(i) $F_e \sub^* S_\aaa$. (ii)   $S_\aaa \sub^* F_e$.  \end{claim}
  \n The claim is verified by induction on $e$. It holds  for $e=0$ because $F_0= S_\ES= \NN$.
   Suppose the claim is true for $e$.  To verify it for  $e+1$, let $\gamma = g \uhr{e+1}$, and  let $s_\gamma$ be the largest stage $s$ such that $\gamma$ is initialized at stage $s$. 
   
\medskip
\n    (i) Suppose $x \in F_{e+1}$. Then $x \in F_e$,  so inductively $x \in S_\aaa$ for almost all such $x$. By construction,  any element $x$ that isn't promoted to $S_\gamma$ is also removed from $F_{e+1}$, unless $x$ is the last element $\aaa$ reserves. However,  in that case necessarily $\gamma = \aaa \ape 2$, so this leads to  at most one new element in $F_{e+1}\setminus S_\gamma$.

\medskip
\n  (ii)     
   Suppose $x \in S_\gamma$. Then $x \in S_\aaa$,  so inductively $x \in F_e$ for almost all such $x$. 
   At stage $s \ge s_\gamma$ an  element $x$ of $S_\aaa$  can not be removed from $F_{e+1}$ by a strategy $\beta >_L\aaa$ because $S_\beta\cap S_\aaa= \ES$ because of  (3) verified above, and since $\beta$ can only remove elements from $S_\beta$. So $x$ can only be  removed by a strategy  tied to~$\aaa$.

If    $   \aaa \colon M_e$ removes    $x$ from $F_{e+1}$, then $x \not \in S'_\aaa$,  contradiction.  
So by construction the only way $x$ can be removed from $F_{e+1}$  is by $\aaa \colon \+ P_e$, which for a sufficiently large  $x$ means that $x$  does not get promoted to   $S_\gamma$ either.
  
  \begin{claim}    $\+ M_e$ is met, namely, $F_e \setminus F_{e+1}$ is infinite. \end{claim}
 \n  To see this, recall  $\aaa = g \uhr e$. By the foregoing claim, the action of $\aaa:M_e$ removes infinitely many elements of $S_\aaa \sub F_e$  from $F_{e+1}$.

  \begin{claim} $\+ P_e$ is met. \end{claim}
  Suppose the hypothesis of $\+ P_e$ holds. Then every number $\aaa$ reserves eventually gets processed. So  either $g(e)= 0$, in which case $F_{e+1} \sub^* V_{e,0}$ by  Claim~\ref{cl:3}, or $g(e)= 1$, in which case $F_{e+1} \sub^* V_{e,1}$ by  Claim~\ref{cl:3}.
 \end{proof}
 
 Added Jan.\ 2020: Lempp, Miller, Nies and Soskova (in prep.) have shown that both $\+ U$ and the analog of the independence number are Medvedev equivalent to the mass problem of dominating functions. The paper
will be posted. 

              \section{Yu Liang}
\subsection{Concerning the Joint coding theorem:  a handwritten proof by Harrington}
The following theorem has appeared in \cite{CY15}.

\begin{theorem}\label{theorem: joint coding}
Given two uncountable $\Sigma^1_1$-sets $A_0, A_1$, for any real $z$, there are reals $x_0\in A_0$ and $x_1\in A_1$ so that $x_0\oplus x_1\equiv_h z\oplus \KO$.
\end{theorem}

On 11/July/2019, Steve Simpson sent us a manuscript by Harrington dated on 9/1975 in which Theorem \ref{theorem: joint coding} was already proved. Harrington's proof is model theoretic and   certainly  deserves to be studied in detail. The manuscript follows.



# The hyperdegrees of reals in products of uncountable $\Sigma_1^1$ sets

Leo Harrington

9/ 75

**Theorem** Let $A_0$, $A_1$ be two uncountable $\Sigma_1^1$ sets of reals ( real = subset of $\omega$ ). For all reals $b \geq_{hyp} \mathcal{O}$ there exists reals $a_0$, $a_1$ s.t. $a_0 \in A_0$, $a_1 \in A_1$ and $b \equiv_{hyp} \langle a_0, a_1 \rangle$.

**proof**

Let $\omega_1 = 1^{st}$ non-recursive ordinal.

Let $\mathcal{B}$ be the smallest collection of sentences s.t.

"$j \in \dot{x}$" and "$j \notin \dot{x}$" are in $\mathcal{B}$ for each $j \in \omega$;

if $X \subseteq \mathcal{B}$ & if $X \in L_{\omega_1}$ then $\bigvee X$ and $\bigwedge X$ are both in $\mathcal{B}$.

For $a$ a real & for $\varphi \in \mathcal{B}$, $a \models \varphi$ iff $\varphi$ is true when $a$ is substituted for $\dot{x}$ in $\varphi$.

By making $A_0$ & $A_1$ smaller if neccessary wma that $a \in A_\varepsilon$ ($\varepsilon = 0$ or $1$) $\Longrightarrow$ $a$ is not hyp & $\mathcal{O} \not\leq_{hyp} a$.

Thus we can find (by the Spector-Gandy Theorem) $T_0, T_1 \subseteq \mathcal{B}$ both $\bar{\Sigma}_1$ over $L_{\omega_1}$ s.t.

for all reals $a$, $a \in A_\varepsilon$ ($\varepsilon = 0$ or $1$) iff: $a \models T_\varepsilon$ and $L_{\omega_1}[a]$ is admissible.

Also $T_\varepsilon \vdash$ "$x$ is not hyp".

———

View integers $e$ as indexing $\Sigma_1$ over $L_{\omega_1}$ subsets, $U(e)$, of $\mathcal{B}$.

Wma that this indexing is s.t.:

for all $\langle V_i \rangle_{i \in \omega}$, a $\Sigma_1$ over $L_{\omega_1}$ enumeration of $\Sigma_1$ over $L_{\omega_1}$ subsets of $\mathcal{B}$, there is $e \in \omega$ s.t.

$\forall i \quad U(\langle e, i \rangle) = V_i$.

also wma that:

Given $e_1, \ldots, e_r$ we can effectively find $e_1 * \ldots * e_r$ s.t. $U(e_1 * \ldots * e_r) = U(e_1) \cup \ldots \cup \ldots$
$$\ldots \cup U(e_r).$$

Let $\mathcal{C}$ be a $\Sigma_1$ formula s.t. for $e \in \omega$,
$\mathcal{C}(e) \equiv \exists X \exists W \exists P (W$ witnesses that $X \subseteq U(e)$
and $P$ is a proof of a contradiction from $X$).

Let $f: \Omega \longrightarrow \omega_1$ be the usual onto func.

For $\varepsilon = 0$ or $1$, let $\bar{\varepsilon} = 1 - \varepsilon$.

For $\eta \in \omega^{<\omega}$ let $\ell(\eta) =$ length of $\eta$.

---

For each $\eta \in \omega^{<\omega}$, each $\alpha \leq \omega_1$ & each $\varepsilon$, we will (attempt to) define an integer $d(\eta, \alpha, \varepsilon)$ as follows:

If $\eta = \emptyset$ then let $d(\eta, \alpha, \varepsilon) = ~\text{~~~~}~ e(\varepsilon)$ (where $e(\varepsilon)$ is some previously chosen index for $T_\varepsilon$).

If $\eta = \tau^\frown m$ and if $\forall \varepsilon ~ d(\tau, \alpha, \varepsilon)$ has been defined, then:

View $\ell(\tau)$ as coding a pair $\langle \varepsilon, e \rangle$.

Let $\hbar(\tau, \alpha) = 1^{st}$ $n \in \omega$ s.t.

$$L_\alpha \not\Vdash C(d(\tau, \alpha, \varepsilon) * \langle e, n \rangle).$$

If there is no such $n$, let $\hbar(\tau, \alpha) = \infty$.

[**Def.** For $t \in 2^{<\omega}$ & for $e \in \omega$, let $e \circ t$ be an index for

$U(e) \cup \{ ~ "j \in \dot{x}" / t(j) = 0 \} \cup \{ ~ "j \notin \dot{x}" / t(j) = 1 \}.$

(i.e., $U(e \circ t) = \ulcorner ~ \urcorner$ ). ]

We define a real $c(\tau, \alpha)$ by:

$c(\tau, \alpha)(k) =$

$$\begin{cases} 0 & \text{if } L_\alpha \not\Vdash C\Big(d(\tau, \alpha, \varepsilon) \circ \big(\overline{c(\tau,\alpha)}(k) ^\frown 0\big)\Big) \\ 1 & \text{o.w.} \end{cases}$$

Since $L_\alpha \not\Vdash C(d(\tau, \alpha, \varepsilon))$, we have,

$\forall k \qquad L_\alpha \not\Vdash C\Big(d(\tau, \alpha, \varepsilon) \circ \big(\overline{c(\tau,\alpha)}(k)\big)\Big).$

We call $\overline{c(\tau, \alpha)}(k)$ a branching pt of $c(\tau, \alpha)$ if:
$L_\alpha \not\Vdash C\Big(d(\tau, \alpha, \varepsilon) \circ \big(\overline{c(\tau,\alpha)}(k) ^\frown 0\big)\Big)$

and $L_\alpha \not\Vdash C\Big(d(\tau, \alpha, \varepsilon) \circ \big(\overline{c(\tau,\alpha)}(k) ^\frown 1\big)\Big).$

For $i \in \omega$, let $t_i(\tau, \alpha) = $ the $i$'th branching pt. of $c(\tau, \alpha)$ (if it exists).

Let $\hat\beta(\tau, \alpha) = 1^{st}$ ord $\leq \alpha$ s.t. $\forall i < \hat{n}(\tau,$

$L_{\hat\beta(\tau,\alpha)} \Vdash C\Big(d(\tau, \alpha, \varepsilon) * \langle e, n \rangle\Big).$

Let $p(\tau, \alpha) = 1^{st}$ $p$ in $\mathcal{O}$ s.t. $f(p) = \hat\beta(\tau, \alpha).$

Let $q(\tau, \alpha) = \langle p(\tau, \alpha), m \rangle.$

If $\hat{\beta}(\tau,\alpha) \not\leq \alpha$, or if $t_{q(\eta,\alpha)}(\tau,\alpha)$ is not defined, then leave $d(\eta,\alpha,0)$ & $d(\eta,\alpha,1)$ undefined.

Otherwise, let $d(\eta,\alpha,\varepsilon) = d(\tau,\alpha,\varepsilon) * \langle e, \hat{n}(\tau,\alpha)\rangle$, and let $d(\eta,\alpha,\bar{\varepsilon}) = d(\tau,\alpha,\bar{\varepsilon}) \circ (t_{q(\eta,\alpha)}(\tau,\alpha)^{\curvearrowright}1)$

---

(Notation: For $\alpha \leq \omega_1$, let $U^\alpha(e) = \{\varphi \in BNL_2 \mid L_2 \vDash \varphi \in U(e)\}$.

Facts ① $d(\eta, \omega_1, \varepsilon)$ is defined for all $\eta$ & $\varepsilon$.

② If $d(\eta, \alpha, \varepsilon)$ is defined, then:
$$\tau \leq \eta \implies d(\tau, \alpha, \varepsilon) \text{ is defined } \&$$
$$U(d(\tau,\alpha,\varepsilon)) \subseteq U(d(\eta,\alpha,\varepsilon)).$$

and: $L_2 \not\vDash C(d(\tau,\alpha,\varepsilon))$.

③ If $\eta$ & $\eta'$ are incomparable, and if $\forall \varepsilon \quad d(\eta, \alpha, \varepsilon)$ & $d(\eta', \alpha, \varepsilon)$ are defined, then:
either $\quad U^\alpha(d(\eta,\alpha,0)) \cup U^\alpha(d(\eta',\alpha,0))$
or $\quad\quad U^\alpha(d(\eta,\alpha,1)) \cup U^\alpha(d(\eta',\alpha,1))$
is inconsistent.

④ If $\lambda \leq \omega_1$ is a limit ordinal, and

if $\forall \bar{e}$ $d(\tau, \lambda, \bar{e})$ is defined, then:

ⓐ for all large enough $\alpha < \lambda$, $\forall \bar{e}$
$d(\tau, \alpha, \bar{e})$ is defined & $d(\tau, \alpha, \bar{e}) = d(\tau, \lambda, \bar{e})$.

~~IA~~ In addition ~~then~~ for all $K \in \omega$

ⓑ for all large enough $\alpha < \lambda$

$\overline{C(\tau, \alpha)}(K) = \overline{C(\tau, \lambda)}(K)$ , and

$\overline{C(\tau, \alpha)}(K)$ is a branching pt of $C(\tau, \alpha)$
iff it ~~is~~ is a branching pt. of $C(\tau, \lambda)$.

[proof~~ⓐ~~]: ⓐ is obvious for $\tau = \emptyset$.

Assume ⓐ is true for $\tau$. We will ~~then~~
prove ⓑ for $\tau$, and then prove ⓐ for $\eta = \tau^\frown m$. Assume $\ell(\tau)$ codes $\langle \bar{e}, e \rangle$

ⓑ: For all $v = 0$ or $1$, $\forall u \in 2^{<\omega}$ s.t. $\ell(u) \leq k$
if $L_\lambda \models C\left(d(\tau, \lambda, \bar{e}) \circ (u^\frown v)\right)$ , then
for all large enough $\alpha < \lambda$,
$$L_\alpha \models C\left(d(\tau, \lambda, \bar{e}) \circ (u^\frown v)\right).$$

But there are only finitely many such
$v, u$ , and for all large enough $\alpha < \lambda$
$d(\tau, \lambda, \bar{e}) = d(\tau, \alpha, \bar{e})$ (by ~~ⓑ~~ⓐ for $\tau$ ),
thus:

for all large enough $\alpha < \lambda$, ~~there is stage~~,
$\forall v = 0 \text{ or } 1$, ~~there~~ $\in 2^{<\omega}$ s.t. $\ell(u) \leq k$,
$L_\alpha \models C(\ d(\tau, \alpha, \bar{\varepsilon}) \circ (u {}^\frown v)\ )$ iff
$L_\lambda \models C(\ d(\tau, \lambda, \bar{\varepsilon}) \circ (u {}^\frown v)\ )$.

This ~~by~~ implies ~~exactly what~~ ④ ⑤ ~~infer~~.

④ for $\eta = \tau {}^\frown n$ :

     assume both $d(\eta, \lambda, \varepsilon)$ & $d(\eta, \lambda, \bar{\varepsilon})$ are
defined.

     Thus $\hat{\beta}(\tau, \lambda) < \lambda$ .

So for all large enough $\alpha < \lambda$, $\forall i < \hat{n}(\tau, \lambda)$,
$L_\alpha \models C(\ d(\tau, \alpha, \varepsilon) * \langle e, \hat{n}(\tau, \lambda) \rangle\ )$.

Since $L_\lambda \models C(\ d(\tau, \lambda, \varepsilon) * \langle e, \hat{n}(\tau, \lambda) \rangle\ )$,
we have $L_\alpha \models C(\ d(\tau, \alpha, \varepsilon) * \langle e, \hat{n}(\tau, \lambda) \rangle\ )$.

Thus $\hat{n}(\tau, \alpha) = \hat{n}(\tau, \lambda)$ & $\hat{\beta}(\tau, \alpha) = \hat{\beta}(\tau, \lambda)$.

~~So so~~ So $q(\eta, \alpha) = q(\eta, \lambda)$ .

Thus by ④ ⑤ for $\tau$,

     for all large enough $\alpha < \lambda$,
$$t_{q(\eta, \alpha)}(\tau, \alpha) = t_{q(\eta, \lambda)}(\tau, \lambda).$$

Thus $d(\eta, \alpha, \varepsilon)$ & $d(\eta, \alpha, \bar{\varepsilon})$ are defined and equal to $d(\eta, \lambda, \varepsilon)$ & $d(\eta, \lambda, \bar{\varepsilon})$. □ ]

⑤ Let $\lambda \leq a_1$ be a limit ordinal, and let $a_0, a_1$ be reals neither of which is in $L_{\lambda+1}$, and let $s \in \omega$.

Assume: for unboundedly many $\alpha < \lambda$ there is $\eta \in \omega^{<\omega}$ s.t. $\ell(\eta) = s$ and s.t. $\forall \varepsilon$ $d(\eta, \alpha, \varepsilon)$ is defined and $a_\varepsilon \notin U^\alpha(d(\eta, \alpha, \varepsilon))$.

Then: there is $\eta \in \omega^{<\omega}$ s.t. $\ell(\eta) = s$ and s.t. $\forall \varepsilon$ $d(\eta, \lambda, \varepsilon)$ is defined & $a_\varepsilon \notin U^\lambda(d(\eta, \lambda, \varepsilon))$.

[ prf of ⑤ : By induction on $s$. If $s = 0$ this is trivial. So assume $s = \sigma + 1$.

By assumption we have $B \subseteq \lambda$ s.t. $B$ is unbounded & s.t. $\forall \alpha \in B$ there is $\eta = \eta(\alpha)$ s.t. $\ell(\eta) = s$ and s.t. $\forall \varepsilon$ $d(\eta, \alpha, \varepsilon)$ is defined & $a_\varepsilon \notin U^\alpha(d(\eta, \alpha, \varepsilon))$.

By ③ $\eta(\alpha)$ is unique.

Since ⑤ is true for $\sigma$, we have $\tau \in \omega^{<\omega}$

s.t. $\ell(\tau) = \sigma$ & $\forall \epsilon \; d(\tau, \lambda, \epsilon)$ is defined
& $a_\epsilon \vDash v^\lambda(d(\tau, \lambda, \epsilon))$.

By ③ $\tau$ is unique.

By ④ ② & ③, wma $\forall \alpha \in B$, $\tau \subseteq \eta(\alpha)$, ~~and~~ and $\forall \epsilon$

$d(\tau, \alpha, \epsilon) = d(\tau, \lambda, \epsilon) \underset{\text{def}}{=} d(\epsilon)$. ~~~~

Let $\sigma$ code $\langle \epsilon, e \rangle$.

Since $c(\tau, \lambda) \in L_{\lambda+1}$, we have that $c(\tau, \lambda) \neq a_{\bar{\epsilon}}$. So there is a $k$ s.t. $\overline{c(\tau, \lambda)}(k) \neq a_{\bar{\epsilon}}$. By ④ ② wma $\forall \alpha \in B$

$\overline{c(\tau, \lambda)}(k) = \overline{c(\tau, \lambda)}(k)$.

Since $a_{\bar{\epsilon}} \vDash v^\alpha(d(\eta(\alpha), \alpha, \bar{\epsilon}))$, $\forall \alpha \in B$, we have that $a_{\bar{\epsilon}} \not\supseteq t_{q(\bar{\eta}, \alpha)}^{\eta(\alpha), \alpha}(\tau, \alpha)$, $\forall \alpha \in B$, and thus $q(\bar{\eta}, \alpha) \leq k$, $\forall \alpha \in B$.

Thus wma that there is a $q$ s.t. $\forall \alpha \in B$ $q(\bar{\eta}, \alpha) = q$, and (since $t_{q(\bar{\eta}, \alpha)}^{\eta(\alpha)} \subseteq \overline{c(\tau, \alpha)}(k)$) by ④ ② wma that $t_q(\tau, \alpha) = t_q(\tau, \lambda)$.

$q = \langle p, m \rangle$. So $\forall \alpha \in B$ $p(\tau, \alpha) = p$, and ~~~~ $\eta(\alpha) = \tau \frown m$. Let $\hat{\beta} = f(p)$. Thus

$\hat{\beta} = \beta(\tau, \alpha)$, $\forall \alpha \in B$.

Let $\hat{n} = 1$st $n \in \omega$ s.t.

$L_{\hat{\beta}} \nvdash C(d(\varepsilon) * \langle e, n \rangle)$;

and let $\hat{n} = \infty$ if there is no such $n$.

By definition of $\beta(\tau, \alpha)$, $\forall \alpha \in B$,

$L_\alpha \nvdash C(d(\varepsilon) * \langle e, \hat{n} \rangle)$. Thus

$L_\lambda \nvdash C(d(\varepsilon) * \langle e, \hat{n} \rangle)$. So $\hat{n}(\tau, \lambda) = \hat{n}$,

$\text{to } \beta(\tau, \lambda) = \hat{\beta}$.

Let $\eta = \tau^\frown m$.

Thus $q(\eta, \lambda) = q$.

Thus $t_{q(\eta, \lambda)}(\tau, \lambda)$ is defined,

$\text{to } \beta(\tau, \lambda) < \lambda$. So $d(\eta, \lambda, \varepsilon)$ & $d(\eta, \lambda, \bar{\varepsilon})$

are both defined and are equal to

$d(\eta, \alpha, \varepsilon)$ & $d(\eta, \alpha, \bar{\varepsilon})$, $\forall \alpha \in B$.

Since $\forall \nu = 0$ or $1$, $a_\nu \in U^\alpha(d(\eta, \lambda, \nu))$

for all $\alpha \in B$, we have

$\qquad a_\nu \in U^\lambda(d(\eta, \lambda, \nu))$. $\square$ ]

Now:

Let $b: \omega \to \omega$.

Let $U_\varepsilon = \bigcup_s U(d(\overline{b}(s), \omega_1, \varepsilon))$.

It is easy to see that $U_\varepsilon$ is a complete consistent extension of $T_\varepsilon$.

Let $a_\varepsilon \underset{def}{=} \{ j \mid "j \in \dot{x}" \in U_\varepsilon \}$.

So $a_\varepsilon \models U_\varepsilon$. By the construction we have: given $e$, if $\forall i \; a_\varepsilon \not\models U(\langle e, i \rangle)$ then $\exists \beta < \omega_1$ s.t. $\forall i \; a_\varepsilon \not\models U^\beta(\langle e, i \rangle)$.

Thus $L_{\omega_1}[a_\varepsilon]$ is admissible & so $a_\varepsilon \in A_\varepsilon$.

Since $a_\varepsilon$ is constructed using over $L_{\omega_1}$ using $b$ as a guide, we have

$$a_\varepsilon \leq_{hyp} \langle b, \mathcal{O} \rangle.$$

Since $b$ is coded into the construction of $a_0, a_1$, we have

$$b \leq_{hyp} \langle a_0, a_1, \mathcal{O} \rangle.$$

Thus,

it suffices to show that $\emptyset \leq_{hyp} \langle a_0, a_1 \rangle$.

By ① ② (applied to $\lambda = \omega_1$), we can find $\alpha(0) < \alpha(1) < \cdots \leq \omega_1$ s.t.

$\langle \alpha(s) \rangle_{s \in \omega}$ is $\Delta_1$ over $L_{\omega_1}[\langle a_0, a_1 \rangle]$,

and s.t. $\forall s$ ~~s~~ $\exists \eta \in \omega^{<\omega}$ s.t. $\ell(\eta) = s$ &

s.t. $\forall \varepsilon$ ( $d(\eta, \alpha(s), \varepsilon)$ is defined &

$a_\varepsilon \models U^{\alpha(s)}(d(\eta, \alpha(s), \varepsilon))$.

Let $\lambda = \sup_s \alpha(s)$. If $\lambda = \omega_1$ then

$L_{\omega_1}[\langle a_0, a_1 \rangle]$ is not admissible and ~~so~~ so $\emptyset \leq_{hyp}$

$\langle a_0, a_1 \rangle$.

If $\lambda < \omega_1$, then since $a_\varepsilon \in A_\varepsilon$,

$a_\varepsilon \notin L_{\lambda + 1}$. Thus by ⑤ & ③ there is

a real $b'$ s.t. $\forall s \, \forall \varepsilon \, d(\overline{b}'(s), \lambda, \varepsilon)$

is defined & $a_\varepsilon \models U^\lambda(d(\overline{b}'(s), \lambda, \varepsilon))$.

But then,

for all $e$, if $a_\varepsilon \nvDash U^\lambda(\langle e, i \rangle)$ for all $i$,

then $\exists \beta < \lambda$ s.t. $a_\varepsilon \nvDash U^\beta(\langle e, i \rangle) \, \forall i$.

Thus $L_\lambda[a_3]$ is admissible, which is absurd.

□

This result was obtained in 6/74, a few days after Martin & Friedman had independently shown:

If $A$ is an uncountable $\Delta^1_1$ set of reals, & if $b \geq_{hyp} a$, then there is $a \in A$ s.t. $a \equiv_{hyp} b$.

Our proof here relies on Friedman's proof of the above (and Friedman's proof in turn relied on our earlier proof of the above for the case $b = a$, (which in turn was motivated by Friedman's even earlier proof of the above for the case when $A$ contains no hyp reals ) ).



\subsection{A note concerning the Borelness of upper cone of hyperdegrees.}

\begin{proposition}
The following statement is  independent from $ZFC$:  for any real $x$, the set $A_x=\{y\mid y\geq_h x\}$ is Borel.
\end{proposition}
\begin{proof}
Clearly if $V=L$, then  the set $A_x=\{y\mid y\geq_h x\}$ is Borel for any real $x$.

Now suppose that $\omega_1=(\omega_1)^L$ and there is an $L$-random real $r$. We prove that $A_r$ is not Borel. Otherwise, there is a countable ordinal $\alpha$ so that $$\forall z (z\in A_r \implies r\in L_{\alpha}[z]).$$

Let $y\in L$ be a real so that $\omega_1^y>\alpha$. By Martin's theorem relative to $y$, there is a $\Delta^1_1(y)$-Sacks generic real $z$ be a so that $z\oplus y\geq_h r$. Then $r$ is $L_{\alpha}[y\oplus z]$-random and so $r\not\in L_{\alpha}[y\oplus z]$ but $r\leq_h y\oplus z$, a contradiction.
\end{proof}

\n {\bf Note}: Actually it can be proved that $A_x$ is Borel if and only if $x\in L$. The idea is as follows: Suppose that $\alpha$ is a countable ordinal and $x\not\in L$. Let $\beta>\alpha$ be a countable admissible ordinal. Then there is a real $y$ so that $\omega_1^y=\beta$ but $x\not\leq_h y$. Then one can prove that there is a real $z\geq_T y$ so that $z>_h y$ but   $x\not\in L_{\beta}[z]$.

   
 \part{Model theory and definability}

      \section{First-order logic, computability, and enriched structures}

A structure in mathematical logic is a non-empty domain (a set)  with some relations and functions defined on it. Examples are graphs, and groups. First-order logic is made for dealing with them: the relations and functions are described by atomics formulas, the quantifiers range over the set. When the domain is countably infinite, the notions of computability also work well with such structures: via some coding, one can see the domain elements as inputs to machines, and one requires that the relations and functions be computable.

Many kinds of  structures that mathematicians study are not of this kind. Often they  study ``enriched structures": a non-empty domain, some relations and functions on it, and something else, some extra structure. Examples are: 
\begin{enumerate} \item metric structures: there is a distance function producing a topology, the relations are closed, the functions continuous. 
\item  Hilbert spaces, with a $\mathbb C$-valued inner product.
\item $C^*$-algebras, which are Banach $*$-algebras  satisfying $||x x^*|| = ||x||^2$. 
\item tracial von Neumann algebras $N$, such as the hyperfinite $\mathrm{II}_1$-factor~$\+ R$. The extra structure is a \emph{trace}, a positive  functional sending $1_N$ to $1$, corresponding to the (normalised)  trace of matrices
\item  Lie groups:  groups on a manifold.
\item Profinite groups: a group on a compact, totally disconnected space with group  operations continuous.
\end{enumerate}
And these are often the most relevant for applications elsewhere: e.g.\ Hilbert spaces and operator algebras are used in quantum physics. Lie groups also play a major role in modern physics. 

\begin{remark} The last two examples   suggest that the proposed   view may be that of a  logician. The extra structure can  also  be viewed as the primary one, with  the algebraic structure sitting on top of it. For instance, Lie groups are group objects in the category of manifolds (which has Cartesian products); these group objects form   a new category where the morphisms are the Lie group homomorphisms.  Ditto for profinite groups, which are group objects on compact totally disconnected topological spaces.   \end{remark}
How do we extend the methods of first-order logic, and of computability, to such structures? 
There are two pathways. 

\smallskip
\n 
A. The first pathway, not followed here, is  to also extend the language, or to adapt  the notion of computability to the new setting. For the language, this has been done in many of the cases above. E.g., continuous logic has been introduced to deal with metric structures, and also with tracial von Neumann algebras. Relations are now real valued, and  the distance function plays the role that formerly  equality had played.
For topological groups, one can use a two-sorted language: one sort for the elements, and another for the open sets, say.

For computability, there are Blum-Shub-Smale (BSS)    machines that directly work on reals, as an example.

\smallskip
\n
B. The second pathway is to     retain the bare tools of   first-order logic and computability  developed for     the naked structures. The approach to the enriched structures  will then necessarily be indirect. We still express things about, or compute with, the elements of the domain with their basic relationships; the expressivity is enhanced only because the semantics is different. For instance, one can  study the  expressivity of first-order logic in profinite groups. If the theory determines the group within a class of groups,  this is called quasi-axiomatisable (QA) relative to the class. Jarden and Lubotzky~\cite{Jarden.Lubotzky:08} (also  see \cite[Thm.\ 4.2.3]{Segal:09}) showed that each topologically f.g.\ profinite group is QA within the profinite groups.   Nies,  Segal and Tent \cite{Nies.Segal.Tent:19} show that for many profinite  groups there is in fact a single sentence determining it up to topological isomorphism within the  class of profinite groups. This property is called  finitely axiomatizable (FA) relative to the class.  Examples of FA groups  include $SL_n(\mathbb Z_p)$ for any odd prime~$p$ not dividing $n$.  

In the computability setting, the problem is that many of the structures are now uncountable, e.g.\ Polish metric spaces. We have to reduce these enriched structures to a countable core, which can be  done by choosing a dense countable subset $D$. In this way we can define computable metric spaces: each element is given a the limit of a certain fast converging Cauchy sequence of elements of $D$. For  profinite groups the most natural way to extend computability to the whole structure is by  asking that the group be  an effective inverse limit of finite groups with onto projections.  This is e.g. the case for the additive group of $p$-adics,  $\ZZ_p = \projlim_n C_{p^n}$. One also gets an ultrametric out of this, relative to which the group operations are computable.   

In the case of the von Neumann hyperfinite   $\mathrm{II}_1$ factor  $\+ R$ it is easy to determine the canonical computable structure as a metric space. $\+ R$ can be introduced as the completion of the $*$-algebra $\bigcup_n M_{2^n}(\mathbb C)$ with respect to norm given by the inner product $\la A ,B \ra = \tau_n (B^* A)$, where $\tau_n$ is the normalized trace given by $\tau_n(C)= \tp{-n} \sum_i C_{ii}$; see e.g.\ the Goldbring 2013 notes~\cite{Goldbring:13}. 

Consider a class $\+ C$ of enriched structures. To structures in $\+ C$ can be   algebraically isomorphic, and fully isomorphic, i.e.\  also the extra structure is preserved. How much stronger is the second of  these notions? For Polish groups (topological groups where the topology is Polish), any Baire measurable isomorphism is continuous. So, if the Axiom of Determinacy (AD) holds (which contradicts the axiom of choice), then there is no difference. 
So, to build profinite groups that are only algebraically isomorphic one needs to use the axiom of choice.   For instance let $A$  be the product of all cyclic groups of order $p^{n}$, $ n \in \NN - \{0\}$.  Then $A$  is  algebraically isomorphic to $A \times  \ZZ_{p}$  by  Cor 3.2 in Kiehlmann (J. Group Theory 16 (2013), 141-157), which uses AC.  They
are not topologically isomorphic because the torsion subgroup is dense in $A$ but not in $A \times \ZZ_p$.  As alternative to Kiehlmann's proof one can use Ulm invariants for abelian p groups.   One can show the dual fact  in abelian torsion groups, that $\bigoplus_{n} C_{p^{n}} \cong\bigoplus_{n}
C_{p^{n}} + C_{p^{\infty}}$. To make Ulm invariants work one again needs AC (in the guise of the  well ordering
theorem).

If an enriched  structure $M$  is quasi-axiomatizable within $\+ C$ then each algebraic isomorphim with another structure in $\+ C$ is automatically a full isomorphism. This can be used to find examples of structures $M$ that are not QA in their class. E.g.\ Kiehlmann's group $A$ is not QA in the profinite groups (assuming AC).

%
%
 
 \def\cprime{$'$} \def\cprime{$'$}

%

\begin{thebibliography}{10}

\bibitem{Bartoszynski.Judah:book}
T.~Bartoszy\'nski and H.~Judah.
\newblock {\em Set Theory. On the structure of the real line}.
\newblock A K Peters, Wellesley, MA, 1995.
\newblock 546 pages.

\bibitem{Bartoszynski:84}
Tomek Bartoszy{\'n}ski.
\newblock Additivity of measure implies additivity of category.
\newblock {\em Trans. Amer. Math. Soc.}, 281(1):209--213, 1984.

\bibitem{Blass:10}
Andreas Blass.
\newblock Combinatorial cardinal characteristics of the continuum.
\newblock In Matthew Foreman and Akihiro Kanamori, editors, {\em Handbook of
  Set Theory}, volume~1, pages 395--489. Springer, Dordrecht Heidelberg London
  New York, 2010.

\bibitem{Brendle.Brooke.ea:14}
J.~Brendle, A.~Brooke-Taylor, Keng~Meng Ng, and A.~Nies.
\newblock An analogy between cardinal characteristics and highness properties
  of oracles.
\newblock In {\em Proceedings of the 13th Asian Logic Conference: Guangzhou,
  China}, pages 1--28. World Scientific, 2013.
\newblock \url{http://arxiv.org/abs/1404.2839}.

\bibitem{CY15}
C.~T. Chong and Liang Yu.
\newblock Randomness in the higher setting.
\newblock {\em J. Symb. Log.}, 80(4):1131--1148, 2015.

\bibitem{Ershov:72}
Y.~Ershov.
\newblock Elementary group theories.
\newblock In {\em Doklady Akademii Nauk}, volume 203, pages 1240--1243. Russian
  Academy of Sciences, 1972.

\bibitem{Felgner:75}
U.~Felgner.
\newblock On $\aleph_0$-categorical extra-special $p$-groups.
\newblock {\em Logique et Analyse}, pages 407--428, 1975.

\bibitem{Fried.Jarden:06}
M.~Fried and M.~Jarden.
\newblock {\em Field arithmetic}, volume~11.
\newblock Springer Science \& Business Media, 2006.

\bibitem{Goldbring:13}
Isaac Goldbring.
\newblock A gentle introduction to von neumann algebras for model theorists,
  2013.

\bibitem{Hirschfeldt.Nies.ea:07}
D.~Hirschfeldt, A.~Nies, and F.~Stephan.
\newblock Using random sets as oracles.
\newblock {\em J. Lond. Math. Soc. (2)}, 75(3):610--622, 2007.

\bibitem{Hodges:93}
W.~Hodges.
\newblock {\em Model Theory}.
\newblock Encyclopedia of Mathematics. Cambridge University Press, Cambridge,
  1993.

\bibitem{Jarden.Lubotzky:08}
M.~Jarden and A.~Lubotzky.
\newblock Elementary equivalence of profinite groups.
\newblock {\em Bulletin of the London Mathematical Society}, 40(5):887--896,
  2008.

\bibitem{Merkle.ea:06}
W.~Merkle, J.~Miller, A.~Nies, J.~Reimann, and F.~Stephan.
\newblock Kolmogorov-{L}oveland randomness and stochasticity.
\newblock {\em Ann. Pure Appl. Logic}, 138(1-3):183--210, 2006.

\bibitem{Morozov.Nies:05}
A.~Morozov and A.\ Nies.
\newblock Finitely generated groups and first-order logic.
\newblock {\em J.\ Lond.\ Math.\ Soc.}, 71(2):545--562, 2005.

\bibitem{Nies:AM}
A.\ Nies.
\newblock Lowness properties and randomness.
\newblock {\em Adv. in Math.}, 197:274--305, 2005.

\bibitem{Nies:DescribeGroups}
A.~Nies.
\newblock Describing groups.
\newblock {\em Bull. Symbolic Logic}, 13(3):305--339, 2007.

\bibitem{Nies:book}
A.~Nies.
\newblock {\em Computability and {R}andomness}, volume~51 of {\em Oxford Logic
  Guides}.
\newblock Oxford University Press, Oxford, 2009.
\newblock 444 pages. Paperback version 2011.

\bibitem{Nies.Scholz:18}
A.~Nies and V.~Scholz.
\newblock Martin-{L}{\"o}f random quantum states.
\newblock {\em Journal of Mathematical Physics}, 60(9):092201, 2019.
\newblock available at doi.org/10.1063/1.5094660.

\bibitem{Nies.Segal.Tent:19}
A.~Nies, D.~Segal, and K.~Tent.
\newblock Finite axiomatizability for profinite groups i: group theory.
\newblock {\em http://arxiv.org/abs/1907.02262, 29 pages}, 2019.

\bibitem{Nies.Semukhin:09}
A.~Nies and P.~Semukhin.
\newblock Finite automata presentable abelian groups.
\newblock {\em Annals of Pure and Applied Logic}, 161(3):458--467, 2009.

\bibitem{Nies.Thomas:08}
A.~Nies and R.~Thomas.
\newblock Fa-presentable groups and rings.
\newblock {\em Journal of Algebra}, 320(2):569--585, 2008.

\bibitem{Nikolov.Segal:07}
N.~Nikolov and D.~Segal.
\newblock On finitely generated profinite groups. {I}. {S}trong completeness
  and uniform bounds.
\newblock {\em Ann. of Math. (2)}, 165(1):171--238, 2007.

\bibitem{Perin.Sklinos:12}
C.~Perin and R.~Sklinos.
\newblock Homogeneity in the free group.
\newblock {\em Duke Mathematical Journal}, 161(13):2635--2668, 2012.

\bibitem{Raisonnier:84}
Jean Raisonnier.
\newblock A mathematical proof of {S}. {S}helah's theorem on the measure
  problem and related results.
\newblock {\em Israel J. Math.}, 48:48--56, 1984.

\bibitem{Rupprecht:thesis}
N.~Rupprecht.
\newblock {\em Effective correspondents to cardinal characteristics in
  {C}icho\'n's diagram}.
\newblock PhD thesis, University of Michigan, 2010.

\bibitem{Rupprecht:10}
N.~Rupprecht.
\newblock Relativized {S}chnorr tests with universal behavior.
\newblock {\em Arch. Math. Logic}, 49(5):555--570, 2010.

\bibitem{Rute:16}
J.~Rute.
\newblock Computable randomness and betting for computable probability spaces.
\newblock {\em Mathematical Logic Quarterly}, 62(4-5):335--366, 2016.

\bibitem{Segal:09}
D.~Segal.
\newblock {\em Words: notes on verbal width in groups}, volume 361.
\newblock Cambridge University Press, 2009.

\bibitem{Slaman.Solovay:91}
T.~Slaman and R.~Solovay.
\newblock When oracles do not help.
\newblock In {\em Proceedings of the fourth annual workshop on Computational
  learning theory}, pages 379--383. Morgan Kaufmann Publishers Inc., 1991.

\bibitem{Soukup:18}
D.~Soukup.
\newblock Two infinite quantities and their surprising relationship.
\newblock {\em arXiv preprint arXiv:1803.04331}, 2018.

\bibitem{Terwijn.Zambella:01}
S.~Terwijn and D.~Zambella.
\newblock Algorithmic randomness and lowness.
\newblock {\em J. Symbolic Logic}, 66:1199--1205, 2001.

\end{thebibliography}
%

\end{document}